\documentclass[11pt]{amsart}
\usepackage{tabularx,booktabs,tikz}
\usepackage{caption}
\usepackage{amsmath}
\usepackage{amsfonts}

\usepackage{amscd}
\usepackage{amsthm}
\usepackage{amssymb} \usepackage{latexsym}
\usepackage{eufrak}
\usepackage{euscript}
\usepackage{epsfig}
\usepackage{graphics}
\usepackage{graphicx}
\usepackage{float}
\usepackage{subfigure}
\usepackage{array}
\usepackage{enumerate}
\usepackage{dsfont}
\usepackage{color}
\usepackage{wasysym}
\usepackage{hyperref}
\usepackage{pdfsync}

\newcommand{\bel}[1]{\begin{equation}\label{#1}}

\newcommand{\be}{\begin{equation}}

\newcommand{\ba}{\begin{eqnarray}}
\newcommand{\ea}{\end{eqnarray}}

\newcommand{\qe}{\end{equation}}

\newcommand{\N}{{\mathbb N}}
\newcommand{\Z}{{\mathbb Z}}
\newcommand{\R}{{\mathbb R}}

\newcommand{\om}{\Omega}

\newcommand{\m}{\mathcal{M}}
\newcommand{\g}{\mathcal{G}}

\newcommand{\su}{\subset}
\newcommand{\dm}{{\delta\m}}
\newcommand{\de}{{\delta}}
\newcommand{\pa}{\partial}
\newcommand{\na}{\nabla}

\newcommand{\Hmm}[1]{\leavevmode{\marginpar{\tiny%
$\hbox to 0mm{\hspace*{-0.5mm}$\leftarrow$\hss}%
\vcenter{\vrule depth 0.1mm height 0.1mm width \the\marginparwidth}%
\hbox to
0mm{\hss$\rightarrow$\hspace*{-0.5mm}}$\\\relax\raggedright #1}}}

\newtheorem{theorem}{Theorem}[section]

\newtheorem{lemma}[theorem]{Lemma}
\newtheorem{corollary}[theorem]{Corollary}
\newtheorem{definition}[theorem]{Definition}

\newtheorem{remark}[theorem]{Remark}

\newtheorem{prop}[theorem]{Proposition}
\newtheorem{problem}[theorem]{Problem}

\newtheorem{example}[theorem]{Example}
\newtheorem{claim}[theorem]{Claim}

\newcommand{\tm}{\begin{theorem}}
\newcommand{\tmd}{\end{theorem}}
\newcommand{\co}{\begin{corollary}}
\newcommand{\cod}{\end{corollary}}
\newcommand{\prp}{\begin{prop}}
\newcommand{\prpd}{\end{prop}}
\newcommand{\pf}{\begin{proof}}
\newcommand{\pfd}{\end{proof}}
\newcommand{\rmk}{\begin{remark}}
\newcommand{\rmkd}{\end{remark}}
\newcommand{\ex}{\begin{example}}
\newcommand{\exd}{\end{example}}
\newcommand{\pr}{\begin{problem}}
\newcommand{\prd}{\end{problem}}

\newcommand{\df}{\begin{definition}}
\newcommand{\dfd}{\end{definition}}
\newcommand{\lm}{\begin{lemma}}
\newcommand{\lmd}{\end{lemma}}

\begin{document}

\title[On area-minimizing subgraphs in integer lattices]{On area-minimizing subgraphs in integer lattices}

\author{Zunwu He}
\address{Zunwu He: School of Mathematics, South China University of Technology, 510641, Guangzhou, China}
\email{hzwmath789@scut.edu.cn}

\author{Bobo Hua}
\address{
Bobo Hua: School of Mathematical Sciences, LMNS, Fudan University, Shanghai 200433, China;  Shanghai Center for Mathematical Sciences, Jiangwan Campus, Fudan University, No. 2005 Songhu Road, Shanghai 200438, China.}
\email{bobohua@fudan.edu.cn}

\begin{abstract}



{We introduce area-minimizing subgraphs in an infinite graph via the formulation of functions of bounded variations initiated by De Giorgi. We classify area-minimizing subgraphs in the two-dimensional integer lattice up to isomorphisms, and prove general geometric properties for those in high-dimensional cases.}
\end{abstract}
\maketitle


\par
\maketitle

\bigskip


\section{Introduction}
Area-minimizing submanifolds in the Euclidean space $\R^n$ are important concepts in geometric measure theory. For the co-dimensional one case in $\R^n$, De Giorgi  \cite{DeGiorgi1961} initiated a formulation using functions of bounded variations, called an area-minimizing hypersurface or a boundary of a subset of least perimeter, see \cite{Giustibook}. A celebrated result is as follows.
\begin{theorem}[Simons \cite{Simons68}] Every area minimizing hypersurface in $\R^n$ for $3\leq n\leq 7$ is flat, i.e. a boundary of a half space.\end{theorem}

The following equation is called the minimal surface equation
\begin{equation}
{\rm div}\left(\frac{\nabla u}{\sqrt{1+|\nabla u|^2}}\right)=0,\quad \R^k.\end{equation}
Bernstein first proved that any entire solution of the minimal surface equation on $\R^2$ is affine, see \cite{Bernstein27}. Such a statement on the triviality of entire solutions is now called the Bernstein theorem.
By the observations of Fleming \cite{Fleming62} and De Giorgi \cite{DeGiorgiBern}, the Bernstein theorem of the minimal surface equation is reduced to the classification of area-minimizing hypersurfaces or area-minimizing cones in the Euclidean space. This leads to the following Bernstein theorem: any entire solution of the minimal surface equation on $\R^k$ is affine if and only if $k \leq 7,$ for which the sharpness follows from the construction of a non-affine solution on $\R^8$ by Bombieri, De Giorgi, and Giusti \cite{BGG69}.

We recall the basic setting in $\R^n.$ A function $f$ is called of locally bounded variations in $\R^n$ if $f\in L^1_{loc}(\R^n)$ and whose distributional derivative is a Radon measure in $\R^n.$
We denote by $BV_{loc}(\R^n)$ the space of functions of locally bounded variations in $\R^n.$  A Lebesgue measurable set $F\subset \R^n$ is called a Caccioppoli set if $1_F\in BV_{loc}(\R^n),$ where $1_F$ is the indicator function on $F.$
 A Caccioppoli set $F$ is called area-minimizing in $\R^n$ if for any open set $\Omega\subset\subset \R^n,$ any Caccioppoli set $W$ with $L^n((W\Delta F)\setminus \Omega)=0,$ $$\int_{\overline{\Omega}}|\nabla 1_F|\leq \int_{\overline{\Omega}}|\nabla 1_W|,$$ where $W\Delta F:=(W\setminus F)\cup (F\setminus W)$, $L^n$ is the Lebesgue measure, and $\int |\nabla \cdot|$ is the BV seminorm. This in fact means that $\partial F$ is an area-minimizing hypersurface. This is equivalent to that for any open set $\Omega\subset\subset \R^n,$ and $g\in  BV_{loc}(\R^n)$ with $g-1_F= 0$ a.e. on $\R^n\setminus\Omega,$$$\int_{\overline{\Omega}}|\nabla 1_F|\leq \int_{\overline{\Omega}} |\nabla g|.$$ Such a function $1_F$ is called of least gradient in $\R^n,$ which has been extensively studied in the literature, see e.g. \cite{BGG69,Miranda67,Sternberg92,Juutinen05,Mazon14,Moradifam17,Jerrard18,Gorny18,Moradifam18,Fotouhi19,Zuniga19}.

Let  $\g=(V,E)$ be a simple undirected graph with the set of vertices $V$ and the set of edges $E.$ Two vertices $x,y$ are called neighbors, denoted by $x\sim y,$ if there is an edge connecting $x$ and $y,$ i.e. $\{x,y\}\in E.$ For any subset $\om\subset V$, denote by $\om^c:=V\backslash\om$ the complement of $\om$, by $\de\om:=\{x\in\om:\exists y\in\om^c,x\sim y\}$ the vertex boundary of $\om$ and by $\tau\om:=\{z\in\om^c:\exists w\in\om,z\sim w\}$ the exterior vertex boundary of $\om$. We write $\bar\om:=\om\cup\tau\om.$ Given any $A,B\su V,$ we set $$E(A,B):=\{\{x,y\}\in E:x\in A,y\in B\}.$$
We say $\pa\om:=E(\om,\om^c)$ is the edge boundary of $\om,$ and set $E_\om:=E(\om,\bar\om)$. We introduce discrete analogs of a subset of least perimeter and an area-minimizing subset.
\df\label{leastperimeter}
For a finite subset $U\subset V,$ a subset $K\su\bar U$ is called of least perimeter in $U$ if for any $\hat K\su\bar U$ with $\hat K\cap\tau U=K\cap\tau U,$ we have $$|\pa K\cap E_U|\leq |\pa \hat K\cap E_U|,$$ where $|\cdot |$ denotes the cardinality of a set.
\dfd \df\label{minimalgraph}
A proper nonempty subset $A\su V$ is called area-minimizing in $V$ if for any finite $U\su V$, $A\cap\bar U$ is of least perimeter in $U.$ In this case, we identify the subset $A$ with its induced subgraph $\g_A$ on $A,$ and call $A$ a minimal subgraph in $V$ for convenience. 
\dfd

\rmk
The subset $A\su V$ is called minimal if and only if for any finite $\om,\hat F\su V$ with $\hat F\triangle F\su\om$, then $$|\pa F\cap E_\om|\leq |\pa \hat F\cap E_\om|.$$
\rmkd

For any $x\sim y,$ we define $$\nabla_{(x,y)}f:=f(y)-f(x).$$
For finite $U\subset V,$ the 1-Dirichlet energy on $U$ is given by, for any $f\in \R^{\overline{U}}$, $$J_U(f):=\frac{1}{2}\sum_{\{x,y\}\in E_U}|\na_{(x,y)}f|.$$

For a subset $\om\subset V,$ any antisymmetric function $a$ on $E_\om,$ i.e. for any $x,y$ with $\{x,y\}\in E_\om,$ $a_{xy}=-a_{yx},$ is called \emph{a current} on $\om.$
For a function $f\in\R^{\overline{\om}},$ we call the current $a$ is \emph{a current associated with $f$ on $\om$} if
$$a_{xy}\in\mathrm{Sgn}(f(x)-f(y)),\quad \forall \ x,y\ \text{with}\ \{x,y\}\in E_{\om},$$ where
\begin{equation}\nonumber
 \mathrm{Sgn}(t):=\left\{
\begin{aligned}
1&, x>0,&\\
[-1,1]&, x=0,&\\
-1&, x<0.&
\end{aligned}
\right.
\end{equation} Let $\mathcal{C}_\om(f)$ be the set of currents associated with $f$ on $\om.$ The 1-Laplacian is defined as a set-valued mapping $\Delta_1:\R^{\overline{\om}}\rightarrow 2^{\R^\om},$
\begin{align}\label{oneLaplacian}
\Delta_1^\om f=\{g\in\R^\om:g(x)=\sum\limits_{y\sim x}a_{xy}, a\in \mathcal{C}_\om(f)\}.
\end{align}
A current $a\in\mathcal{C}_\om(f)$ is called \emph{minimal} in $\om$ if $\sum\limits_{y\sim x}a_{xy}=0$ for all $x\in \om.$
The following is well-known using the results in convex analysis \cite{Rockafellar70,RockWets98}, see e.g. Chang \cite{Chang16} and Hein-B\"uhler \cite{Hein-Buhler10}.
\begin{prop}\label{prop:subdiff} For finite $U\subset V$ and a given function $\varphi\in \R^{\tau U},$ the subdifferential of the functional $J_U$ at $f$ in $\{f\in \R^{\overline{U}}: f|_{\tau U}=\varphi\}$ is given by $\Delta_1^U f.$
\end{prop}



The discrete analog of functions of least gradient was introduced in metric random walk spaces including graphs by Maz\'{o}n, P\'{e}rez-Llanos, Rossi, and Toledo \cite{Mazon16}, see also \cite{GornyMazon21,Mazon23book,MST23}.
\df\label{leastgradient}

For finite $U\subset V,$ a function $f\in\R^{\bar U}$ is called of least gradient in $U$ if for any $g\in\R^{\bar U}$ with $g|_{\tau U}=f|_{\tau U}$, then $$J_U(g)\geq J_U(f).$$
\dfd

The following are the characterizations of minimal subgraphs.
\tm\label{minimalcurrent1}
For $A\subset V,$ the following are equivalent:
\begin{enumerate}
\item $A$ is a minimal subgraph in $V.$
\item For any finite $U\subset V,$ $1_{A\cap \overline{U}}$ is of least gradient in $U.$
\item $0\in \Delta^V_1(1_A).$
\end{enumerate}
\tmd

Now we turn to minimal subgraphs in the integer lattices. We denote by $\Z^n$ the graph of the $n$-dimensional integer lattice consisting of the set of vertices $\Z^n=\{x\in \R^n: x_i\in \Z, \forall 1\leq i\leq n\}$ and the set of edges $\{\{x,y\}: |x-y|=1, x,y\in\Z^n\}.$

The first main result is the following statement.
\tm\label{main1}
Any minimal subgraph in $\Z^2$ is exactly one of the following subgraphs up to isomorphisms:
\begin{enumerate}[(1)]
\item 7 families of minimal subgraphs with non-geodesic boundary (see Section 2 for the definition); see Fig.\ref{fig1-1}-Fig. \ref{fig1-7}.\label{main1-1}
\item 3 families of minimal subgraphs with geodesic non-simple boundary (see Section 2 for the definition); see Fig.\ref{fig2-1}-Fig.\ref{fig2-3}.
\item  19 families of minimal subgraphs with geodesic simple boundary. More precisely, there are 7 families of connected minimal subgraphs: 5 fimilies all have exactly one connected component of boundaries; see Fig.\ref{fig3-1-1}-Fig.\ref{fig3-1-5}; 2 families both have exactly two connected of boundaries; see Fig.\ref{fig3-2-1}, Fig.\ref{fig3-2-2}. There are 12 families of disconnected minimal subgraphs: they are all complementary subgraphs of connected minimal subgraphs.
\end{enumerate}
\tmd

\begin{figure}[htbp]
\centering
\begin{minipage}{0.49\linewidth}
\centering
       \includegraphics[height=0.35\linewidth,width=0.5\linewidth]{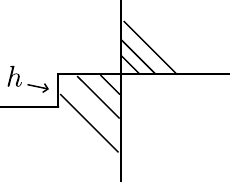}
  \caption{
  \small
 $h\leq 2$. }
 \label{fig1-1}
\end{minipage}
\begin{minipage}{0.49\linewidth}
\centering
       \includegraphics[height=0.35\linewidth,width=0.5\linewidth]{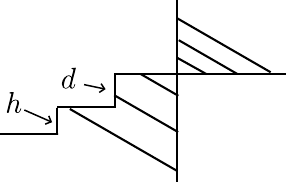}
  \caption{
  \small
 $h=d=1.$  }
\label{fig1-2}
\end{minipage}
\begin{minipage}{0.49\linewidth}
\centering
       \includegraphics[height=0.35\linewidth,width=0.5\linewidth]{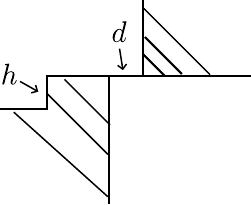}
  \caption{
  \small
 $h=d=1.$ }
 \label{fig1-3}
\end{minipage}
\begin{minipage}{0.49\linewidth}
\centering
       \includegraphics[height=0.35\linewidth,width=0.5\linewidth]{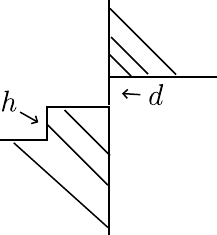}
  \caption{
  \small $h=d=1.$   }
\label{fig1-4}
\end{minipage}

\begin{minipage}{0.49\linewidth}
\centering
       \includegraphics[height=0.35\linewidth,width=0.5\linewidth]{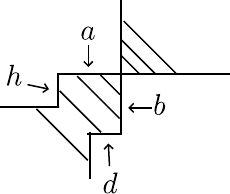}
  \caption{
  \small $h=d=1,a\geq 2,b\geq 2.$ }
 \label{fig1-5}
\end{minipage}
\begin{minipage}{0.49\linewidth}
\centering
       \includegraphics[height=0.35\linewidth,width=0.5\linewidth]{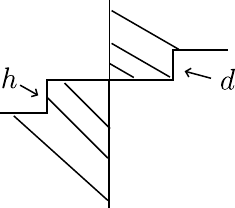}
  \caption{
  \small $h=d=1.$   }
\label{fig1-6}
\end{minipage}

\centering
\begin{minipage}{0.49\linewidth}
\centering
       \includegraphics[height=0.35\linewidth,width=0.5\linewidth]{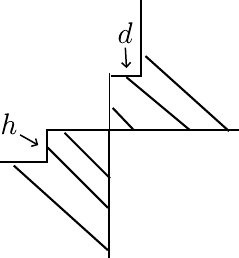}
  \caption{
  \small $h=d=1.$ }
 \label{fig1-7}
\end{minipage}
\end{figure}

\begin{figure}[htbp]
\centering
\begin{minipage}{0.49\linewidth}
\centering
       \includegraphics[height=0.65\linewidth,width=0.8\linewidth]{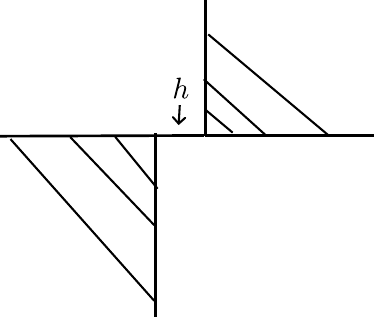}
  \caption{
  \small $h\leq 2.$}
 \label{fig2-1}
\end{minipage}
\begin{minipage}{0.49\linewidth}
\centering
       \includegraphics[height=0.65\linewidth,width=0.8\linewidth]{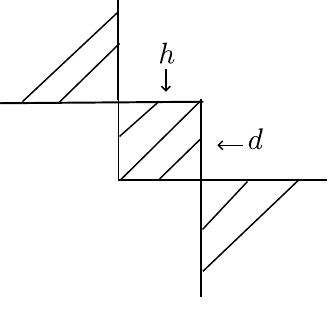}
  \caption{
  \small $h=d=1.$  }
\label{fig2-2}
\end{minipage}

\begin{minipage}{0.49\linewidth}
\centering
       \includegraphics[height=0.65\linewidth,width=0.8\linewidth]{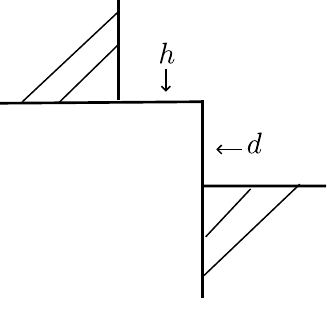}
  \caption{
  \small $h=d=1.$  }
\label{fig2-3}
\end{minipage}
\begin{minipage}{0.49\linewidth}
\centering
       \includegraphics[height=0.65\linewidth,width=0.8\linewidth]{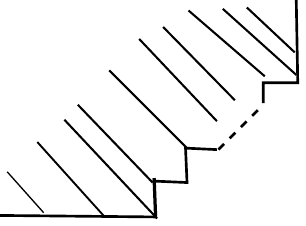}
  \caption{
  \small }
 \label{fig3-1-1}
\end{minipage}
\begin{minipage}{0.49\linewidth}
\centering
       \includegraphics[height=0.65\linewidth,width=0.8\linewidth]{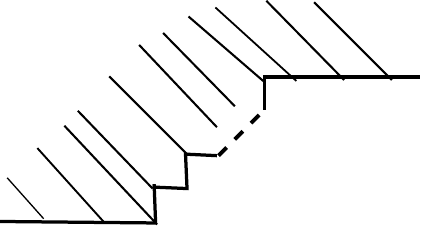}
  \caption{
  \small }
\label{fig3-1-2}
\end{minipage}
\begin{minipage}{0.49\linewidth}
\centering
       \includegraphics[height=0.65\linewidth,width=0.8\linewidth]{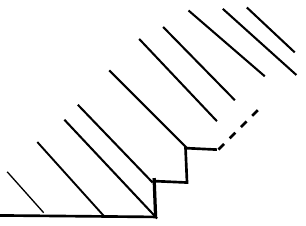}
  \caption{
  \small }
 \label{fig3-1-3}
\end{minipage}


\begin{minipage}{0.49\linewidth}
\centering
       \includegraphics[height=0.65\linewidth,width=0.8\linewidth]{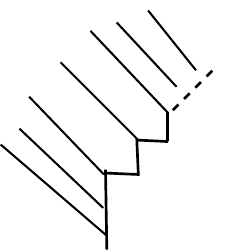}
  \caption{
  \small   }
\label{fig3-1-4}
\end{minipage}
\begin{minipage}{0.49\linewidth}
\centering
       \includegraphics[height=0.65\linewidth,width=0.8\linewidth]{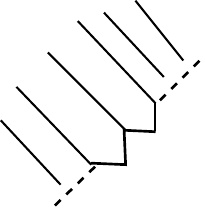}
  \caption{
  \small  }
\label{fig3-1-5}
\end{minipage}
\end{figure}

\begin{figure}
\begin{minipage}{0.49\linewidth}
\centering
       \includegraphics[height=0.6\linewidth,width=0.8\linewidth]{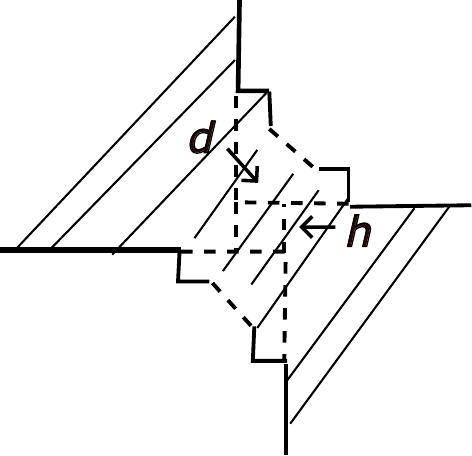}
  \caption{ $0\leq d\leq h+2$
  \small }
 \label{fig3-2-1}
\end{minipage}
\begin{minipage}{0.49\linewidth}
\centering
       \includegraphics[height=0.6\linewidth,width=0.8\linewidth]{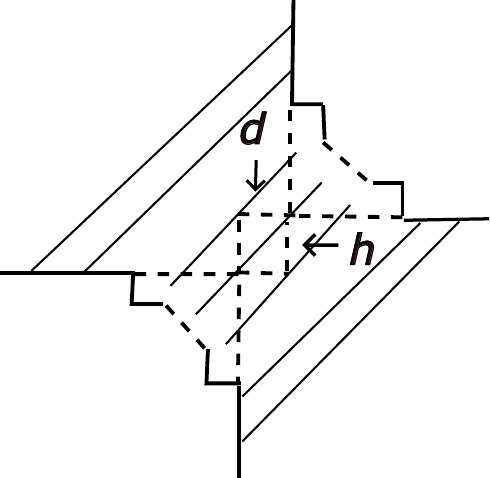}
  \caption{$d\geq 0,h\geq 0$
  \small  }
\label{fig3-2-2}
\end{minipage}
\end{figure}

The classification of minimal subgraphs is quite complicated. The subtleties mainly lie in two aspects: 1. The boundaries of minimal subgraphs may not be geodesic. 2. The boundary of minimal subgraphs may not be simple. For the first aspect, we find that all boundaries of minimal subgraphs consist of geodesic rays instead of geodesic lines; for example see Claim \ref{geodesicraybdy}. For the second aspect, the boundary of any minimal subgraph contains at most one unit square and distributes diagonally in some way; see Corollary \ref{noloops} and Lemma \ref{noT}. The arguments are based on topological, combinatorial and coarsely geometric methods.

The proof strategies of Theorem \ref{main1} are as follows. There are three key observations for the proof: Corollary \ref{convexclosed2}, Lemma \ref{orientedboundary}, and Lemma \ref{noboundedstrip}. Corollary \ref{convexclosed2} yields some geodesic convexity of minimal subgraphs, and implies that the boundary of any minimal subgraph determines this minimal subgraph itself in some sense. So the crucial point is to characterize the boundary of any minimal subgraph. Lemma \ref{orientedboundary} is a useful tool to describe the geometry of oriented boundary of any minimal subgraph, which suggests that the boundary behaves like a simple path in general. Lemma \ref{noboundedstrip} gives a strong geometric restriction of the boundary of any minimal subgraph, which applies to deduce a strong property, Corollary \ref{nomoretwocomponents}, that the number of connected components of a minimal subgraph and its boundary is at most two. By the results in Lemma \ref{complementaryminimal} and Lemma \ref{nomoretwocomponents}, it suffices to consider connected minimal subgraphs with at most two connected boundary components. There are two steps: First, we use several important properties to describe the geometry and topology of the boundary of any minimal subgraph  (whether the boundary is geodesic or simple). Second, we use the arguments of currents to confirm the precise structure of boundary and the minimal subgraph; see Corollary \ref{minimalcurrent} and Lemma \ref{orientedgeodesic}.

On the other hand, the classification of high-dimensional minimal subgraphs(i.e. in $\Z^n,n\geq 3$) seems much more complicated. The reasons are as follows: The result as in Lemma \ref{orientedboundary} fails in higher dimension, see Fig.\ref{interiortransversal};  the analogs of Corollary \ref{convexclosed2} and Lemma \ref{noboundedstrip} in higher dimension don't provide enough information about the local geometry.

\begin{figure}
\begin{minipage}{0.6\linewidth}
\centering
       \includegraphics[height=0.6\linewidth,width=0.8\linewidth]{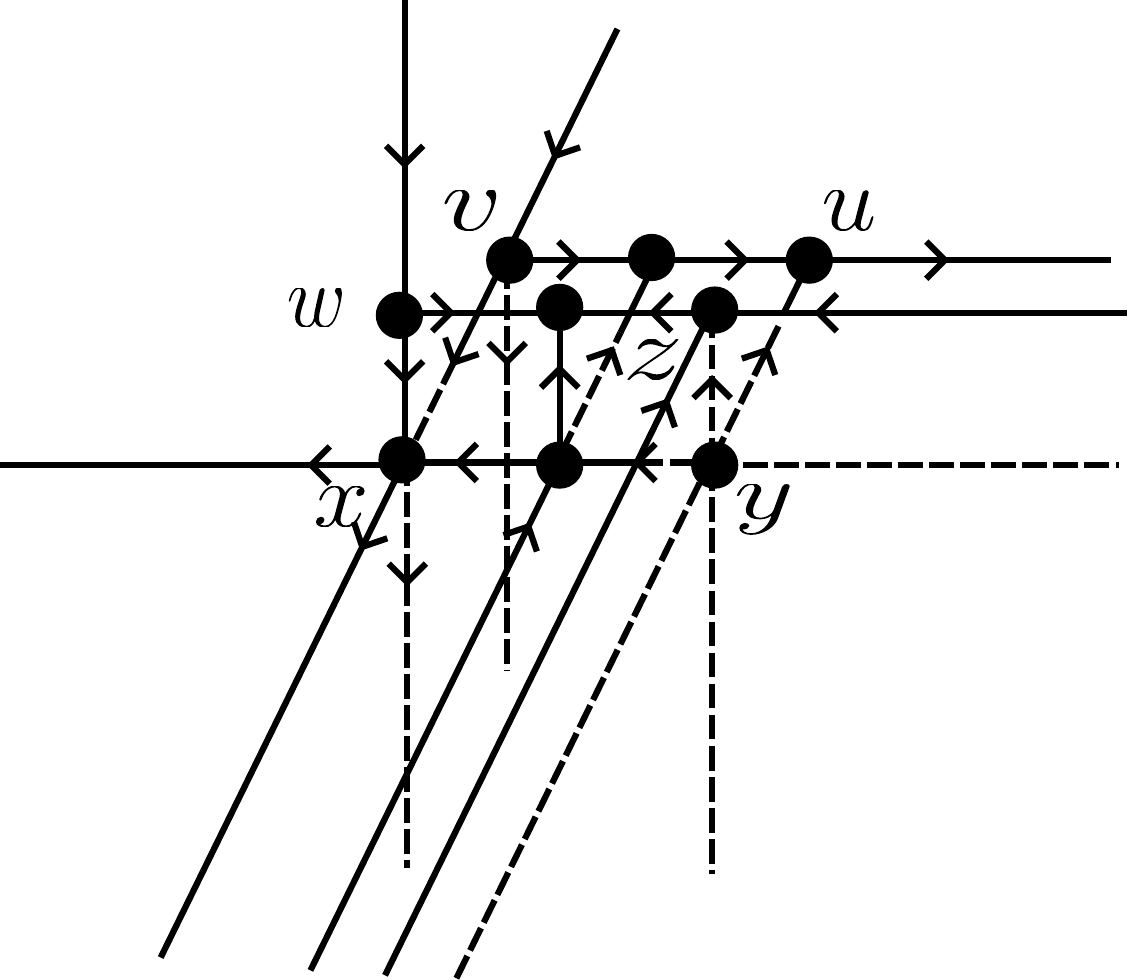}
  \caption{
  \small }
 \label{interiortransversal}
\end{minipage}
\end{figure}

Nevertheless, we can decompose any three dimensional minimal subgraph $\m$ into a minimal sub-subgraph $\m^3\cup \m^2$ (union of three skeleton and two skeleton) and its one skeleton.
\tm\label{redution3dim}
For a minimal subgraph $\m\subset\Z^3$, $\m^3\cup\m^2$ is minimal. 
\tmd

\rmk
$\m^3$ is not minimal in general, see Fig. \ref{M3notmini} for a counterexample by observing $0\notin\Delta_1(1_{\m})(x).$
\begin{figure}
\centering
\begin{minipage}{0.6\linewidth}
\centering
       \includegraphics[height=0.6\linewidth,width=0.9\linewidth]{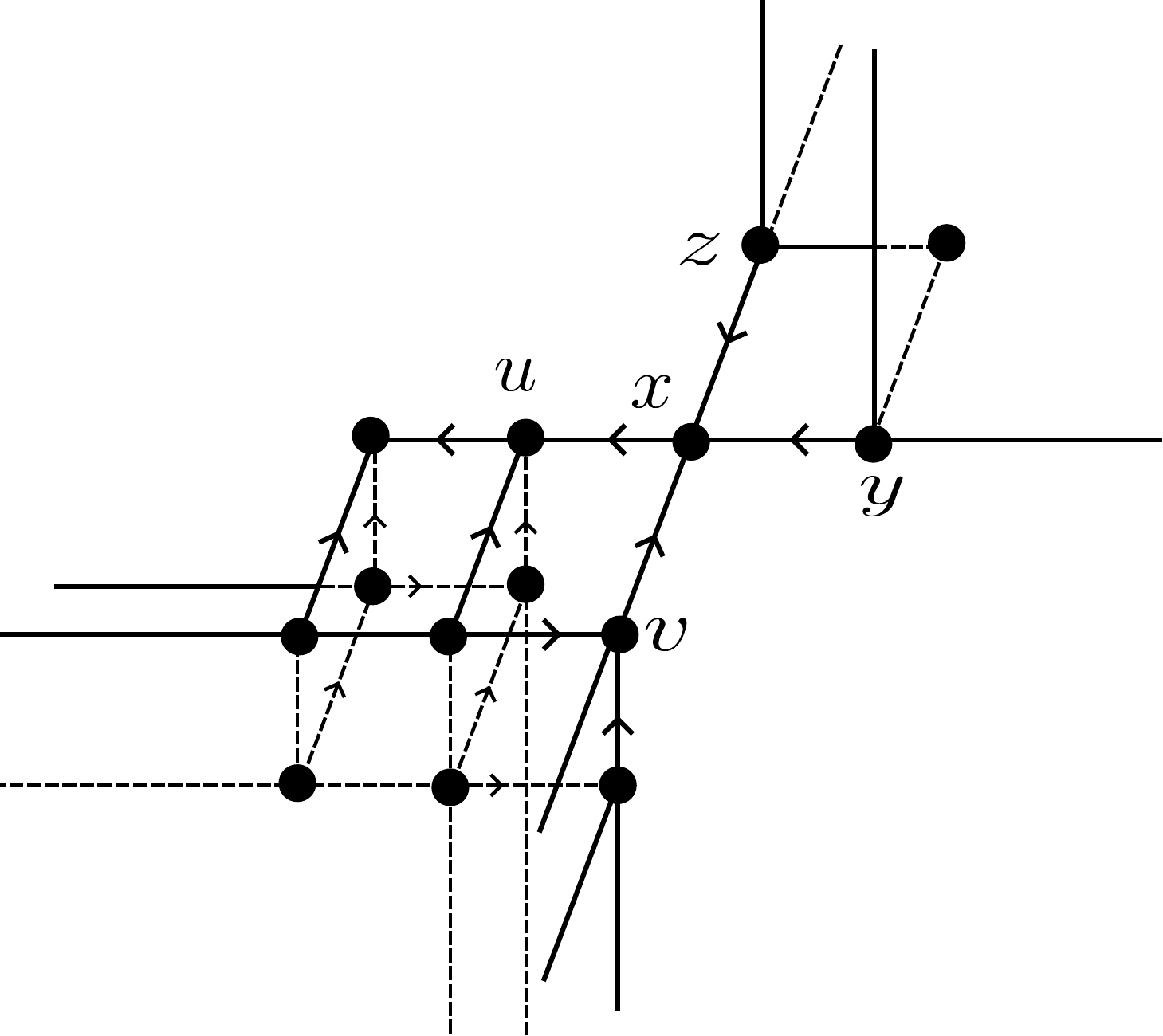}
  \caption{
   }
 \label{M3notmini}
\end{minipage}
\end{figure}
\rmkd

Recall that $\phi:(X,d_X)\longrightarrow(Y,d_Y)$ is called a rough isometry if $$d_X(x_1,x_2)=d_Y(\phi(x_1),\phi(x_2)),d_Y(Y,\phi(X))\leq c,$$ for any $x_1,x_2$ in $X$ and some positive constant $c$. Note that this is stronger than usual definitions in \cite{Woess00,MR1835418}.
For a minimal subgraph $\m\subset\Z^n$, we know $\m^n$ does not always inherit the minimality of $\m$. But it shares the coarse geometry of $\m$ and its boundary. We prove the following result.
\tm\label{ncellroughlyisometry}
For a minimal subgraph $\m\subset\Z^n$, the natural embedding $\m^n$ (with the induced metric) $\subset\m$ is a rough isometry between metric spaces. Moreover, there exists a positive constant $c(n),$ depending only on $n,$ such that for sufficiently large $r$ and any $x\in \m,$
\begin{align}
&\dfrac{|\m^n\cap\hat B_r(x)|}{|\m\cap\hat B_r(x)|}\geq c(n),\label{ncellroughlyisometry1}
\\& \dfrac{|\de\m^{n}\cap\hat B_r(x)|}{|\de \m\cap\hat B_r(x)|}\geq \dfrac{1}{1+2n},\label{ncellroughlyisometry2}
\end{align} where $\hat B_r(x):=\{y\in \R^n:\max\limits_{1\leq i\leq n}|y_i-x_i|\leq r\}$ denotes the $\infty$-normed ball centered at $x$ of radius $r.$
\tmd

We expect that $\m^n$ determines the asymptotic geometry of $\m$ and its boundary, so that the inequalities (\ref{ncellroughlyisometry1}), (\ref{ncellroughlyisometry2}) in Theorem \ref{ncellroughlyisometry} could be possibly improved, see Problem \ref{probelm1}.

Moreover, we prove some restriction on the geometry of $n$ dimensional minimal subgraphs.
\tm\label{noboundedplane}
If $\m\subset\Z^n$ is minimal, then there exist no two parallel hyperplanes bounding $\m$.
\tmd

We remark that the conclusions in Section 4 are independent of those of Section 3, {and note that the maximum principle holds for minimal subgraphs in $\Z^n$ in some sense; see Proposition \ref{maximumprinciple}.}

The organization of this paper is as follows. In Section 2, we introduce some basic concepts and properties of minimal subgraphs. In Section 3, we focus on two dimensional minimal subgraphs and prove Theorem \ref{main1}. In Section 4, we study the geometry of high-dimensional minimal subgraphs and prove Theorem \ref{redution3dim}. In Section 5, we list some open problems on geometry and topology of high-dimensional minimal subgraphs.

\section{preliminaries}
Let $\g=(V,E)$ be a simple, undirected graph. For each edge $\{x,y\}\in E,$ we write $(x,y)$ and $(y,x)$ for associated directed edges. We say $\mathcal{G}$ is connected if for any $x,y\in V$, there is a path $x=x_0\sim x_1\sim \cdots \sim x_n=y$ connecting $x$ and $y$ for some $n\in\N.$ For $x\in V,$ we denote by $\deg(x)$ the vertex degree of $x$ in $V.$  Usually, we consider the subgraph induced on a subset $\m,$ for which the degree of a vertex refers to that in $\m.$ The combinatorial distance $d_{\g}$ or simply $d$ unless specially stated on the graph is defined as, for any $x,y\in V$ and $x\neq y,$ $$d(x,y):=\inf\{n\in\N: \exists\{x_{i}\}_{i=1}^{n-1}\subset V, x\sim x_1\sim\cdots \sim x_{n-1}\sim y \}.$$

{For any function $\varphi\in \R^{\tau\Omega},$ we consider the functional $J_{\Omega,\varphi}:\R^\Omega\to\R$ with Dirichlet boundary condition $\varphi$ given by $$J_{\Omega,\varphi}(g)=J_{\Omega}(\widetilde g),$$ where $\widetilde g\in \R^{\overline{\Omega}}$ such that $\widetilde g|_\Omega=g, \widetilde g|_{\tau\Omega}=\varphi.$}

Note that $f\in\R^{\overline{\Omega}}$ is of least gradient if and only if
$f|_\Omega$ is the minimizer of
$J_{\Omega,f|_{\tau\Omega}}.$ One readily sees that
if $\Omega'\subset \Omega,$ $f$ is of least gradient in $\Omega,$ then $f$ is of least gradient in $\Omega'.$ A similar result holds for sets of least perimeter.


The following is the discrete co-area formula, see \cite{Bar17}.
\begin{prop} For any function $f\in \R^{\overline{\Omega}},$
$$J_\Omega(f)=\int_{-\infty}^\infty |\partial\{f>t\}\cap E_\Omega|dt,$$
where $\{f>t\}:=\{x\in \overline{\om}: f(x)>t\}.$
\end{prop}

Now we prove the following proposition.
\begin{prop}\label{prop:finiteversion} For finite $\Omega\subset V$ and $K\subset \overline\om,$ the following are equivalent:
\begin{enumerate}
\item $K$ is of least perimeter in $\om.$
\item $1_{K}$ is of least gradient in $\om.$
\item $0\in \Delta^\om_1(1_K).$
\end{enumerate}
\end{prop}
\begin{proof}
(1)$\Longrightarrow$(2):
Consider any $g\in \R^{\overline{\Omega}}$ with $g|_{\tau\Omega}=\mathds{1}_K|_{\tau\Omega}.$
For any $t\in (0,1),$ $\{g>t\}\cap {\tau\Omega}=K\cap {\tau\Omega}.$
Since $K$ is of least perimeter in $\om,$ $$|\partial\{g>t\}\cap E_\Omega|\geq|\partial K\cap E_\Omega|.$$

By the co-area formula,
\begin{eqnarray*}
J_\Omega(g)&=&\int_{-\infty}^\infty |\partial\{g>t\}\cap E_\Omega|dt\geq \int_{0}^1 |\partial\{g>t\}\cap E_\Omega|dt\\
&\geq& |\partial K\cap E_\Omega|=J_\Omega(1_K).
\end{eqnarray*}

(2)$\Longrightarrow$(1): For any $\widetilde{K}$ with $K\cap \tau\Omega=\widetilde{K}\cap \tau\Omega,$ the result follows from the least gradient property of $1_K$ by choosing $g=1_{\widetilde{K}}.$

(2)$\Longleftrightarrow$(3):
$1_K\in\R^{\overline{\Omega}}$ is of least gradient in $\om$ if and only if
$1_K|_\Omega$ is the minimizer of
$J_{\Omega,1_K|_{\tau\Omega}}.$ Since $J_{\Omega,1_K|_{\tau\Omega}}$ is a convex function on $\R^\om,$ $1_K|_\Omega$ is the minimizer if and only if $0$ is in the subdifferential of $J_{\Omega,1_K|_{\tau\Omega}}$ at $1_K|_\Omega,$ which is $0\in \Delta_1^\om(1_K)$ by Proposition~\ref{prop:subdiff}.

\end{proof}


We write $B_r(x):=\{y\in V: d(y,x)\leq r\}$ for the ball of radius $r$ centered at $x.$
\begin{proof}[Proof of Theorem~\ref{minimalcurrent1}]
(1)$\Longleftrightarrow$(2): This follows from Proposition~\ref{prop:finiteversion}.

(2)$\Longrightarrow$(3): Consider a sequence of balls $\{B_r\}_{r=1}^\infty$ where $B_r:=B_r(p)$ for some $p\in V.$ For $U=B_r,$ $r\geq 1,$ it follows from Proposition~\ref{prop:finiteversion}, $$0\in \Delta_1^{B_r}(1_{A\cap \overline{B_r}}).$$
That is, for each $r\geq 1,$ there is a minimal current $a^r$ associated with $1_{A\cap \overline{B_r}}$ on $B_r.$ Since there are countable edges in $E$ and $$\sup_{\{x,y\}\in E_{B_r}}|a^r_{xy}|\leq 1,\quad \forall r\geq 1,$$ there is a subsequence $r_i\to \infty$ and a current $a^\infty$ on $V$ such that
$$a^{r_i}_{xy}\to a^\infty_{xy},\quad \forall \{x,y\}\in E.$$
One easily sees that $a^\infty\in \mathcal{C}_V(1_A)$ and $a^\infty$ is minimal. Hence $0\in \Delta_1^V(1_A).$

(3)$\Longrightarrow$(2): For any minimal $a\in \mathcal{C}_V(1_A),$ one easily verifies that for any finite $U\subset V,$
$a|_{E_U}\in \mathcal{C}_U(1_{A\cap \overline{U}}),$ which is also minimal in $U.$

This proves the theorem.
\end{proof}


{Now we introduce some notions on $\Z^2$.} For $x_1,x_2,\cdots,x_k\in \Z^2$, we say these vertices are \emph{horizontal (vertical,resp.)} or $x_1,x_2,\cdots,x_{i}$ is \emph{horizontal  (vertical,resp.) to} $x_{i+1},\cdots,x_k$ if they are in a horizontal (vertical,resp.) line. We say $x$ is \emph{left-horizontal (right-horizontal,up-vertical,down-vertical,resp.) to} or a \emph{left-horizontal (right-horizontal,up-vertical,down-vertical,resp.)} neighbor of $y$ if $x$ is horizontal (vertical,resp.) to $y$ and is on the left(right,up,down,resp.) of $y$.

Given any two paths $\alpha\supset\alpha_1:=x_1\sim x_2\sim x_3$, $\alpha$ is called \emph{flat} if its vertices are all horizontal or vertical. The vertex $x_2$ is called \emph{a corner} of $\alpha$ if the subpath $\alpha_1$ is not flat. We say a path is \emph{simple} if all vertices have one or two neighbors in the path.

Given any path $\alpha$ containing $x$, $x$ is called a \emph{projective horizontal (vertical, resp.) interior point} if there is a horizontal (vertical,resp.) line $\beta$ such that the distance projection (with respect to $\beta$) image of $x$ lies in the interior of that of $\alpha$; see Fig. \ref{figinterior}. Assume the path $\alpha:=x_0\sim x_1\sim x_2\cdots\sim x_k\sim x_{k+1}\subset \m$, we call $\alpha$ is \emph{an isolated path} in $\m$ if $\deg(x_i)=2$ for $1\leq i\leq k$ for the induced subgraph $\m$; see Fig. \ref{figisolatedpath}.
\begin{figure}[htbp]
\centering
\begin{minipage}{0.7\linewidth}
\centering
       \includegraphics[height=0.3\linewidth,width=0.63\linewidth]{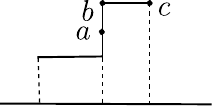}
  \caption{
  \small
 $a,b$ are horizontal interior points, but $c$ is not an interior point. }
 \label{figinterior}
\end{minipage}

\begin{minipage}{0.7\linewidth}
\centering
       \includegraphics[height=0.45\linewidth,width=0.7\linewidth]{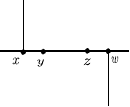}
  \caption{
  \small
 The horizontal path $[x,w]$ is an isolated path in $\m$. }
 \label{figisolatedpath}
\end{minipage}
\end{figure}


By Theorem~\ref{minimalcurrent1}, we have the following.
\co\label{minimalcurrent}
$\m\subset\Z^n$ is minimal if and only if $0\in \Delta_1(1_{\m})$. \cod

\rmk\label{minimalcurrent2}
Corollary \ref{minimalcurrent} provides an effective way to check the minimality of subgraphs by the equation (\ref{oneLaplacian}). This plays an important role to classify all minimal subgraphs in $\Z^2$.

\rmkd


Given a minimal graph $\g=(V,E)\subset \Z^2$, we associate it with a geometric space $X(\g)\subset\R^2$. To be precise, we identify the graph structure with the natural corresponding 1-skeleton in $\R^2$, i.e. identify $(x,y)\in E$ with $[x,y]\subset\R^2$, and associate any loop $x\sim y\sim z\sim w\sim x$ with a unit square enclosed by the loop. The boundary $\de\g$ is called \emph{simple} if for any $x\in\de\g$, there are at most two neighbors of $x$ in $\de\g$. $\de\g$ is called \emph{geodesic} if $d_{\de\g}(y,z)=d_{\g}(y,z)$ holds for all $y,z\in\de\g.$

For $(x,y)\in E,$ the edge $(x,y)$ is called a \emph{boundary edge} of $\g$ if there is at most one unit square in $X(\g)$ containing $[x,y]$. A vertex $z\in v$ is called a \emph{boundary vertex} if it is contained by a boundary edge. A boundary path $\alpha:=x_0\sim x_1\cdots\sim x_k$ is called \emph{oriented in $\g,$} oriented in short, if $k\geq 2$ and $\g$ contains one of the following subgraphs for any subpath $\hat \alpha:=u\sim v\sim w\subset\alpha$; see Fig. \ref{figoriented}.
If we equip $\alpha$ with an orientation, $\alpha$ is called \emph{right (left,resp.) oriented} if these subgraphs for the subpaths lie on the right (left,resp.) hand side of $\alpha$; see Fig. \ref{figoriented2}.
\begin{figure}[htbp]
\centering
\begin{minipage}{0.7\linewidth}
\centering
       \includegraphics[height=0.3\linewidth,width=0.63\linewidth]{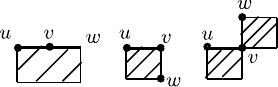}
  \caption{
  \small
 There are three local subgraphs for black bold boundary subpath of length two. }
 \label{figoriented}
\end{minipage}

\begin{minipage}{0.7\linewidth}
\centering
       \includegraphics[height=0.33\linewidth,width=0.4\linewidth]{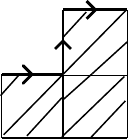}
  \caption{
  \small
 The black bold boundary path is right oriented. }
 \label{figoriented2}
\end{minipage}

\end{figure}





\section{geometry of two dimensional minimal subgraphs and proof of Theorem \ref{main1} }

In this section, we write $\m$ for a minimal subgraph in $\Z^2.$

\lm\label{complementaryminimal}
If $\m$ is minimal, then so is the complementary subgraph $\m^c$.
\lmd
\pf
This is direct by definition and Corollary \ref{minimalcurrent}.
\pfd

\lm\label{nolongline}
$\dm$ can not contain an isolated geodesic path with length $L\geq 3$.
\lmd

\pf
If not, we may assume $\alpha:=x_0\sim x_1\sim\cdots\sim x_k$ is an isolated geodesic path in $\dm$ with $k\geq 3$. Then one can remove this path except $x_0,x_k$, and obtain the new subgraph $\m_1$. Taking $\om=\alpha$, one easily sees that $$|\partial\m\cap E_\om|\geq|\partial\m_1\cap E_\om|+2.$$ This is impossible since $\m$ is minimal.

\pfd

\lm\label{noisolatedpoint}
$\m$ can not contain an isolated point(i.e.\ it has only one neighbor in $\m$). As a consequence, $\dm$ is locally one of the three subgraphs and any boundary vertex has at least two boundary neighbors in $\de\m$; see Fig. \ref{figngh}.
\begin{figure}[htbp]
\centering
\begin{minipage}{0.8\linewidth}
\centering
       \includegraphics[height=0.4\linewidth,width=0.9\linewidth]{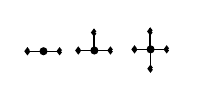}
  \caption{
  \small
 There are three local subgraphs for the bigger black boundary vertex. }
 \label{figngh}
\end{minipage}
\end{figure}
\lmd

\pf
If not, we get another subgraph $\m_1$ by deleting the isolated point $x$. It is clear that $|\partial\m\cap E_\om|=|\partial\m_1\cap E_\om|+2$, where $\om:=\{x\}$. This contradicts the minimality of $\m.$
\pfd

\lm\label{convexclosed}
If there is a finite non-closed simple path in $\m$ lying on one side of a horizontal/vertical line and the line contains two endpoints of the path, then the domain enclosed by the line and the path is contained in $\m$.
\lmd

\pf
We argue by contradiction.

We may assume the path $\alpha:=x_0\sim x_1\sim \cdots\sim x_n$ lies above on the horizontal line $l$ such that $x_0,x_1\in l$. Denote by $\mathcal A$ the domain enclosed by the line and the path.

If $\mathcal A$ is not in $\m$, then one can find $y_0\sim y_1\sim\cdots\sim y_m$ such that the domain enclosed by the horizontal line $l^\prime$ through $y_0,y_m$ (except the line $l^\prime$) is contained in $\m$, and $y_0,y_m\in \m,y_1,\cdots,y_{m-1}\notin \m.$ 
One can get another subgraph $\m_1$ by adding $y_1,\cdots,y_{m-1}$ to $\m$, then one easily deduces that
$$|\partial\m\cap E_\om|\geq|\partial\m_1\cap E_\om|+1+m-(m-1)=|\partial\m_1\cap E_\om|+2.$$
This contradicts the minimality of $\m$.

\pfd

\co\label{convexclosed2}
If $x,y\in\m$ lie in a horizontal/vertical line and in the same connected component of $\m$, then the horizontal/vertical segment between $x$ and $y$ is
contained in $\m.$
\cod
\pf
Since $x,y$ are in the same connected component of $\m$, then there is a finite simple path $\alpha:=x=x_0\sim x_1\cdots\sim x_n=y$ between $x$ and $y$. Let $l$
be the line on which $x,y$ lie, and $\{x=z_0,z_1,\cdots z_k=y\}:=l\cap\alpha$. Note that each subpath between $z_i$ and $z_{i+1}$ of $\alpha$ satisfies the
condition of Lemma \ref{convexclosed}. Thus we prove the result by Lemma \ref{convexclosed}.
\pfd

\co\label{twocrossbd}
Any horizontal/vertical line cannot intersect two boundary edges and one other edge consecutively in a connected component of $\m,$ i.e. there cannot exist two boundary edges $(x_1,x_2),(y_1,y_2)$ and one other edge $(z_1,z_2)$ satisfying that
$x_1,y_1,z_1$ and $x_2,y_2,z_2$ lie in two horizontal/vertical lines in the given order.
\cod
\pf
If not, one can deduce that the rectangle $\mathcal R(x_1,x_2,z_1,z_2)$ is contained in
$\m$ by Corollary \ref{convexclosed2}. This is impossible, since $(y_1,y_2)$ is a boundary edge.

\pfd


\lm\label{orientedboundary}
If there is a finite simple right (left,resp.) oriented boundary path $\alpha:=x_0\sim x_1\sim\cdots\sim x_n $ in $\dm$, and $x_k$ is a projective horizontal (vertical,resp.) interior point in $\alpha$ with $1\leq k\leq n-1$. Then there is no vertex in $\m$ on the left(right,resp.) hand side of $\alpha$ such that it is adjacent and vertical (horizontal,resp.) to $x_k$ in $\m$.
\lmd
\pf
Otherwise, we may assume there is a vertex $y\sim x_k\in\m$ on the left side of $\alpha$ with $y\notin \alpha$. By Lemma \ref{noisolatedpoint}, there is a vertex $y_1(\neq x_k)\sim y\in\m$.

Case 1. If $x_{k-1},x_k,x_{k+1}$ are horizontal/vertical, then it is clear $y,x_k$ are vertical/horizontal. Furthermore, $y_1,x_k$ are also vertical/horizontal. If not, we can assume $y_1,x_{k+1}$ are vertical/horizontal. Combining Corollary \ref{convexclosed2} and the condition that $\alpha$ is left oriented, one can deduce that $(x_k,x_{k+1})$ is contained in two unit squares in $\m$ and it is not a boundary edge of $\m$. It is a contradiction. Applying Lemma \ref{noisolatedpoint} again, there is a vertex $y_2(\neq y_1)\sim y_1\in\m$. By the same argument, we have that $y_2,x_k$ are vertical/horizontal. Continuing the process, we finally get an isolated vertical/horizontal ray $x_k\sim y\sim y_1\sim y_2\sim\cdots$ in $\m$. This contradicts the minimality of $\m$ by Lemma \ref{nolongline}.

Case 2. If $x_{k-1},x_k,x_{k+1}$ are not horizontal/vertical, we may assume that $x_{k-1},x_k$ are horizontal and $x_k,x_{k+1}$ are vertical. Recall that $x_k$ is projective interior vertex in $\alpha$, so that we may assume $x_k,x_s$ are vertical and $x_s,x_{s+1}$ are horizontal for some $k+1\leq s\leq n-1$. If $x_k,y$ are horizontal, then $x_k,x_s$ and $y,x_{s+1}$ are both vertical. Combining Corollary \ref{convexclosed2} with the condition that $\alpha$ is left oriented, one can deduce that $(x_s,x_{s+1})$ is not a boundary edge of $\m$. It is a contradiction. Hence $x_k,y$ are vertical. Using same arguments in Case 1, one can get a vertical ray $x_k\sim y\sim y_1\sim y_2\sim\cdots$ in $\m$. This contradicts the minimality of $\m$ by Lemma \ref{nolongline}.
\pfd

\rmk\label{orientedboundary2}
The proof of Lemma \ref{orientedboundary} is valid for more general case and it is used frequently throughout this section.
\rmkd

\lm\label{noboundedstrip}
 $\dm$ can not contain two parallel horizontal or vertical rays.
\lmd

\pf
If not, we may assume that $l_1:=x_1\sim x_2\sim \cdots,l_2:=y_1\sim y_2\sim\cdots\subset \dm$ are two horizontal rays originating from two vertical vertices $x_1,y_1$ respectively with $c:=d(x_1,y_1)$. Denote by $\mathcal S$ the strip bounded by $l_1,l_2$. So one can find $x_i\in l_1,y_i\in l_2$ subjected to $\hat c:=d(x_1,x_i)=d(y_1,y_i )\geq c+3$, where $i$ is some positive integer. There are two cases.

Case 1. $l_1,l_2$ are in the same connected component of $\m$. Then $\m$ contains $\mathcal S$ by Corollary \ref{convexclosed2}.  One can obtain another graph $\m_1$ via removing the rectangle $\mathcal R(x_2,x_{i-1},y_2,y_{i-1})$. Note that $x_2,y_2,x_3,y_3,\cdots,x_{i-1},y_{i-1}$ have edges not in $\m$ by Lemma \ref{orientedboundary}. Therefore, this yields that
$$|\partial\m\cap E_\om|\geq|\partial\m_1\cap E_\om|+2(\hat c-2)-2c\geq|\partial\m_1\cap E_\om|+2,$$
where $\om=\mathcal R(x_1,x_i,y_1,y_i)$. This is a contradiction by the minimality of $\m.$

Case 2. $l_1,l_2$ are in different connected components of $\m$. We may assume that $l_1$ is above $l_2$.
For each $x_k,y_k$ with integer $k\geq 1$, there exist $[\hat x_k,\hat y_k]\subset [x_k,y_k]$ with
\begin{align}
[\hat x_k,\hat y_k]\cap \m=\{\hat x_k,\hat y_k\},d(\hat x_k,\hat y_k)\geq 2.\label{innerboundary}
\end{align}
Let $\alpha_1:=\{\hat x_1,\hat x_2,\cdots\}(\alpha_2:=\{\hat x_1,\hat x_2,\cdots\},resp.)$.
Then we get another graph $\m_1$ by filling a rectangle $\mathcal R(x_1,x_i,y_1,y_i)$. Note that these $\hat x_k,\hat y_k$ are boundary vertices of $\m$.
Therefore, by using (\ref{innerboundary}) one can show that $$|\partial\m\cap E_\om|\geq|\partial\m_1\cap E_\om|+2\hat c-2c\geq|\partial\m_1\cap E_\om|+4,$$
where $\om=\mathcal R(x_1,x_j,y_1,y_j)$. This contradicts the minimality of $\m.$
\pfd

\co\label{noboundedstrip2}
Any two disjoint infinite simple boundary paths in $\de\m$ can not be contained in some unbounded infinite strip isometric to $[0,d]\times[0,\infty)$, where $d$ is some positive constant.
\cod

\pf
The proof is similar to that of Lemma \ref{noboundedstrip}.  
Assume that $l_j$ is the line through $x_j,y_j$. We may assume
$\beta_1:=z_1\sim z_2\sim\cdots,\beta_2:=w_1\sim w_2\sim\cdots\subset \dm$ are two disjoint infinite boundary simple paths bounded by two horizontal rays $\alpha_1:=x_1\sim x_2\sim \cdots,\alpha_2:=y_1\sim y_2\sim\cdots$ and $x_1,z_1,w_1,y_1$ are vertical. Let $\gamma_j$ be the vertical line through $x_j,y_j$, $c:=d(x_1,y_1),\hat c:=d(x_1,x_i)\geq c+3$ for some integer $i$.

We only prove the case that $\beta_1,\beta_2$ are in the same connected component of $\m,$ and the others are similar. It follows that the strip $\hat{\mathcal S}$ bounded by $\beta_1,\beta_2$ is contained in $\m$ by Corollary \ref{convexclosed2}. Note that there is at least one edge not in $\m$ for $\gamma_j\cap\beta_i(i=1,\ 2)$ with $j\geq 2$ by Corollary \ref{orientedboundary}. Therefore, we have
$$|\partial\m\cap E_\om|\geq|\partial\m_1\cap E_\om|+2(\hat c-2)-2c\geq|\partial\m_1\cap E_\om|+2,$$
where $\om=\hat{\mathcal S}\cap\mathcal{R}(x_2,x_{i-1},y_{i-1},y_2)$ and $\m_1$ is obtained by removing $\om$ from $\m$. This contradicts the minimality of $\m.$
\pfd

\lm\label{digonaldistrubution}
$\m$ doesn't contain one of three subgraphs with $[u,y]\cup[y,z]\cup[z,w]\subset\de\m$; see Fig. \ref{figdiagonal}.   
\begin{figure}[htbp]
\centering
\begin{minipage}{0.8\linewidth}
\centering
       \includegraphics[height=0.25\linewidth,width=\linewidth]{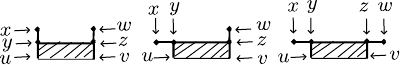}
  \caption{
  \small
$x\sim y\sim u,z\sim w\sim v$. }
 \label{figdiagonal}
\end{minipage}
\end{figure}
\lmd
\pf
By Corollary \ref{convexclosed2}, we have that for the first subgraph in Fig. \ref{figdiagonal} the rectangle $R(x,y,z,w)\subset \m$ and then $[y,z]$ is not a boundary path. It is a contradiction. For the second one in Fig. \ref{figdiagonal}, applying Lemma \ref{orientedboundary} to the projective vertical interior point $y$ in the oriented path $[u,y]\cup[y,z]\cup[z,w]$, this yields that $[y,z]$ or $[u,y]$ is not a boundary path. It is a contradiction.

For the third one in Fig. \ref{figdiagonal}, we have
\begin{claim}\label{noT}
The connected component $\mathcal C$ containing $x$ of $\m$ and the interior of the domain $II$ (or $III,IV$) are disjoint; see Fig. \ref{figdiagonal2}.
\end{claim}

\begin{figure}[htbp]
\centering
\begin{minipage}{0.8\linewidth}
\centering
       \includegraphics[height=0.3\linewidth,width=0.7\linewidth]{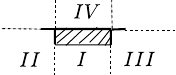}
  \caption{
  \small
}
 \label{figdiagonal2}
\end{minipage}
\end{figure}
\pf[Proof of Claim \ref{noT}]
By Corollary \ref{convexclosed2} and Corollary \ref{orientedboundary}, we deduce that $\mathcal C$ and the interior of the domain $IV$ are disjoint.
By symmetry of the domian $II$ and $III$, one can assume that there is a vertex $s\in\mathcal C$ lying in the interior of the domain $II$. By definition of $\mathcal C$, it is easy to obtain that there is a vertex $p$ (distinct to $x$ or $u$ respectively) vertical or horizontal to $x$ or $u$ respectively. Then using Corollary \ref{convexclosed2} again, we get that $[p,x]\subset\m$ or $[p,u]\subset\m$ and therefore $\m$ contains a unit square containing $x,y,u$. This is impossible since $[y,u]$ is a boundary edge.
\pfd
Assume that $\mathcal C_1$ is the connected component of $\m$ in the domain $I$. Now we consider the boundary $\de \mathcal C_1$.

Case 1. $\de \mathcal C_1$ is finite. Using Lemma \ref{noisolatedpoint}, we can find a closed simple path $\alpha\subset \de \mathcal C_1$ to enclose $\mathcal C_1$. By Corollary \ref{convexclosed2}, we have $\alpha=\de \mathcal C_1$ and that $\mathcal C_1$ is the domain enclosed by $\alpha$. Taking $\om=\mathcal C$, by comparing the edge boundary of $\mathcal C_1$ with that of $[y,z]$, we get that 
$$|\partial\m\cap E_\om|>|\partial\m_1\cap E_\om|+2,$$
where $\m_1$ is obtained by removing $\mathcal C$ except $[y,z]$ from $\m$. This is impossible by the minimality of $\m.$

Case 2. $\de \mathcal C_1$ is infinite. Applying Corollary \ref{noboundedstrip2}, we have that $\de \mathcal C$ contains at most one infinite half simple path. If $\de \mathcal C$ contains one infinite half simple path $\beta$, then by Corollary \ref{twocrossbd} one deduces that $\beta$ only have finite backtracks in the direction of the line containing $x,y$. Hence one can get that the subpath $\beta_1$ which is sufficiently far away from $x$ of $\beta$ is a geodesic ray. This contradicts the minimality of $\m$ by Lemma \ref{nolongline}.
\pfd

\lm\label{squareloop}
If there is a simple loop in $\de\m$, then the loop is a unit square.
\lmd

\pf
Note that the domain enclosed by any closed simple loop in $\de \m$ is contained in $\m$ by Corollary \ref{convexclosed2}. Let $\alpha$ be a simple loop in $\de\m$ and $\mathcal D\subset\m$ be the domain enclosed by $\alpha.$ We may assume that $\mathcal A$ is the smallest rectangle containing $\mathcal D$. Using Corollary \ref{convexclosed2}, we have $\mathcal A\cap\alpha=[a,d^\prime]\cup [a^\prime,b]\cup[b^\prime,c]\cup[c^\prime,d]$, where $[a,d^\prime],[a^\prime,b],[b^\prime,c],[c^\prime,d]$ are the highest, the leftmost, the lowest, the rightmost subpath of $\alpha.$

Since $\alpha$ is a simple loop and $\mathcal A\cap\alpha=[a,d^\prime]\cup [a^\prime,b]\cup[b^\prime,c]\cup[c^\prime,d]$ for the smallest rectangle $\mathcal A$ containing $\mathcal D,$ we can assume that $a,a^\prime,b,b^\prime,c,c\prime,d,d^\prime$ are arranged counterclockwise on the loop $\alpha$. Let $\alpha_1,\alpha_2,\alpha_3,\alpha_4$ be the boundary subpaths connecting $a$ and $a^\prime$, $b$ and $b^\prime$, $c$ and $c^\prime$, $d$ and $d^\prime$ in $\alpha$, respectively.

We claim that $\alpha_1,\alpha_2,\alpha_3,\alpha_4$ are geodesic. It suffices to prove the result for $\alpha_1:=a=a_1\sim a_2\sim \cdots a_i=a^\prime$. Suppose it is not true, we may assume that $a_j,a_{j+1}$ and $a_{j+2},a_{j+3}$ are horizontal, $a_j,a_{j+3}$ and $a_{j+1},a_{j+2}$ are vertical. If $a_{j+1}$ is on the left of $a_j$, then by the fact that $[a^\prime,b]$ is the leftmost, there is a boundary edge $(a_k,a_{k+1})$ for some integer $j+3<k<i-1$ such that $a_j,a_{j+3},a_k$ and $a_{j+1},a_{j+2},a_{k+1}$ are vertical. This is impossible by Corollary \ref{twocrossbd}. If $a_{j+1}$ is on the right of $a_j$, then the square $R(a_j,a_{j+1},a_{j+2},a_{j+3})$ is contained in $\m$ by Corollary \ref{convexclosed2}. This implies $(a_{j+1},a_{j+2})$ is not a boundary edge and it is a contradiction. Thus we prove the claim.

It follows from the above claim that $a=a^\prime,b=b^\prime,c=c^\prime,d=d^\prime$ if and only if $a,a^\prime;b,b^\prime;c,c^\prime;d,d^\prime$ are pairwise horizontal or vertical. Note that $\mathcal D$ have other neighbors in $\m$ exactly adjacent to the subset $\mathcal B\subset \{a,a^\prime,b,b^\prime,c,c^\prime,d,d^\prime\}$ by Lemma \ref{orientedboundary}.

Case 1. At least three of $a\neq a^\prime,b\neq b^\prime,c\neq c^\prime,d\neq d^\prime$ hold. It is clear that $|\mathcal B|\leq 1$ by Lemma \ref{orientedboundary}. This obviously yields that $$|\partial\m\cap E_\om|-|\partial\m_1\cap E_\om|\geq 6\cdot2-2>0,$$
where $\om:=\mathcal D$ and the new subgraph $\m_1$ is obtained by removing $\mathcal D$ from $\m$ except $\mathcal B$. This contradicts the minimality of $\m.$

Case 2. Two or one of $a\neq a^\prime,b\neq b^\prime,c\neq c^\prime,d\neq d^\prime$ hold. We only need to prove the result for the subcase that one of $a\neq a^\prime,b\neq b^\prime,c\neq c^\prime,d\neq d^\prime$ holds, since the other case is similar to those of Case 1 and the previous subcase. Now we may assume $a\neq a^\prime$ and $b= b^\prime,c=c^\prime,d= d^\prime$. Applying Lemma \ref{orientedboundary} and Lemma \ref{digonaldistrubution}, we have $|\mathcal B|\leq 2$ and $\mathcal B=\{b,d\}$ if $|\mathcal B|=2$. Therefore we get $$|\partial\m\cap E_\om|-|\partial\m_1\cap E_\om|\geq3\cdot2-2\cdot2>0,$$
where $\om:=\mathcal D$ and the new subgraph $\m_1$ is obtained by removing $\mathcal D$ from $\m$ except $\mathcal B$. This contradicts the minimality of $\m.$

Case 3. None of $a\neq a^\prime,b\neq b^\prime,c\neq c^\prime,d\neq d^\prime$ holds, i.e. $a=a^\prime,b= b^\prime,c=c^\prime,d= d^\prime$. Then $\mathcal D=\mathcal R(a,b,c,d)$ is a rectangle with $d(a,b)=m\geq 1,d(a,d)=n\geq 1$. By Lemma \ref{orientedboundary}, Lemma \ref{digonaldistrubution}, $|\mathcal B|\leq 2$ and $\mathcal B=\{a,c\}$ or $\mathcal B=\{b,d\}$. Observing that $\m$ is minimal, we deduce that $$|\partial\m_1\cap E_\om|\geq|\partial\m\cap E_\om|\geq|\partial\m_1\cap E_\om|+2(m+n)-4,$$
where $\om:=\mathcal D$ and the new subgraph $\m_1$ is obtained by removing $\mathcal D$ from $\m$ except $\mathcal B$. This implies $m=n=1$ and the result follows.

\pfd

\co\label{squareloop2}
Given any connected component $\mathcal C$ of $\m$, if $\de\mathcal C$ contains some loops, then these loops are exactly one unit square.
\cod

\pf
Otherwise, we may assume that there are two unit distinct squares $\mathcal R(x,y,z,w)$, $\mathcal R(\hat x,\hat y,\hat z,\hat w)$ in $\de\mathcal C$ by Lemma \ref{squareloop}. Suppose that $x,y,z,w$ and $\hat x,\hat y,\hat z,\hat w$ are arranged anticlockwise on $\mathcal R(x,y,z,w),\mathcal R(\hat x,\hat y,\hat z,\hat w)$ respectively, and $(x,w)((\hat x,\hat w),(x,y),(\hat x,\hat y),resp.)$ are the leftmost (rightmost, lowest, highest, resp.) in $\mathcal R(x,y,z,w)$, $\mathcal R(\hat x,\hat y,\hat z,\hat w)$ respectively.

Recall that $\mathcal R(x,y,z,w)$, $\mathcal R(\hat x,\hat y,\hat z,\hat w)$ are in the connected component $\mathcal C$. Applying Corollary \ref{twocrossbd}, we may assume that $d(z,\hat z)=d(\mathcal R(x,y,z,w),\mathcal R(\hat x,\hat y,\hat z,\hat w))$ and there is a simple boundary path $\alpha$ realizing $d(z,\hat z)$ in the domain $I$; see Fig. \ref{figonesquare}. 
\begin{figure}[htbp]
\centering
\begin{minipage}{0.7\linewidth}
\centering
       \includegraphics[height=0.5\linewidth,width=0.6\linewidth]{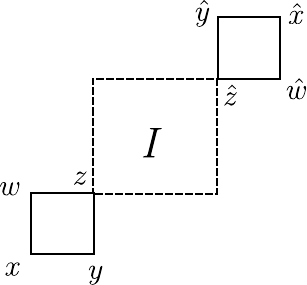}
  \caption{
  \small
The rectangular region enclosed by dotted segments is $I$.}
 \label{figonesquare}
\end{minipage}
\end{figure}

Observe that $w,\hat y$ both have exactly two neighbors in $\m$, by the proof of Lemma \ref{orientedboundary} for the simple boundary path $x\sim w\sim z\cup\alpha\cup \hat z\sim \hat y\sim\hat x$ with right oriented edges $(x,w),(w,z),(\hat z,\hat y),(\hat y,\hat x)$. Similar arguments yield that $y,\hat w$ both have exactly two neighbors in $\m$.

Let $\mathcal R$ be the rectangle containing $w,x,y,\hat w,\hat x,\hat y$ and $\m_1$ be the subgraph obtained by removing the part of $\mathcal C$ in $\mathcal R$ except $x,\hat x$. Using Lemma \ref{twocrossbd} and recalling that $y,w,\hat y,\hat w$ all have exactly two neighbors in $\m$, we have $$|\partial\m\cap E_\om|\geq|\partial\m_1\cap E_\om|+8-4>|\partial\m_1\cap E_\om|,$$ where $\om:=\mathcal R$. This contradicts the minimality of $\m.$

\pfd

\lm\label{orientedgeodesic}
Assume that $\alpha$ is an oriented boundary geodesic line in $\m$, then there are two geodesics (denoted by $\infty$ if the geodesic does not exist) $\alpha_1,\alpha_2$ such that minimal currents in the region $D$ enclosed by $\alpha_1$ and $\alpha_2$ is determined by $\alpha$; see Fig. \ref{figorientedgeodesic1}-\ref{figorientedgeodesic4}. Moreover, the region $\m^\prime:=D\cup \m$ is also minimal and there is an oriented boundary geodesic line $\alpha^\prime$ if $\alpha_1,\alpha_2$ are not both $\infty$.
\begin{figure}[htbp]
\centering
\begin{minipage}{\linewidth}
\centering
       \includegraphics[height=0.3\linewidth,width=0.94\linewidth]{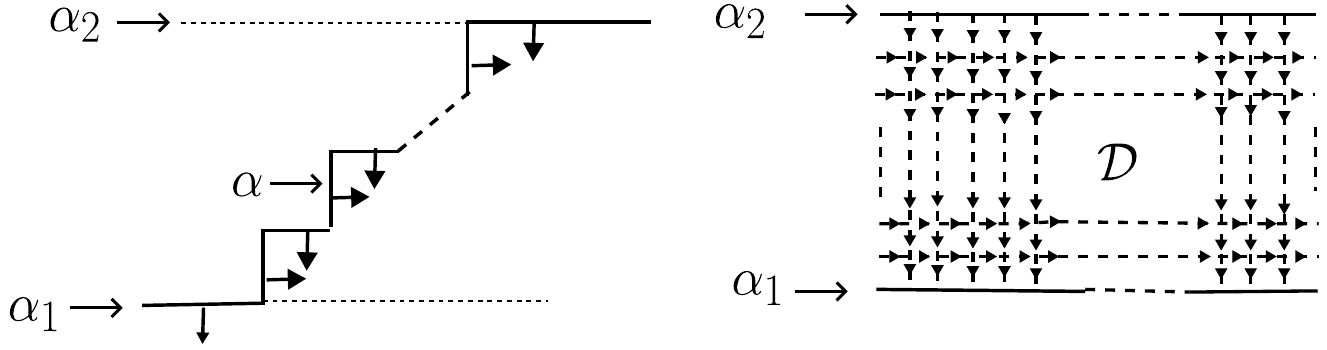}
  \caption{
  \small
$\alpha$ has finitely many corners and the current-determined region (determined by $\Delta_1(1_{\m})=0$) $\mathcal D$ is an unbounded strip. The black arrows indicate the currents (gradients) determined by $1_{\m}$. 
}
 \label{figorientedgeodesic1}
\end{minipage}
\end{figure}

\begin{figure}[htbp]
\centering
\begin{minipage}{\linewidth}
\centering
       \includegraphics[height=0.3\linewidth,width=1.14\linewidth]{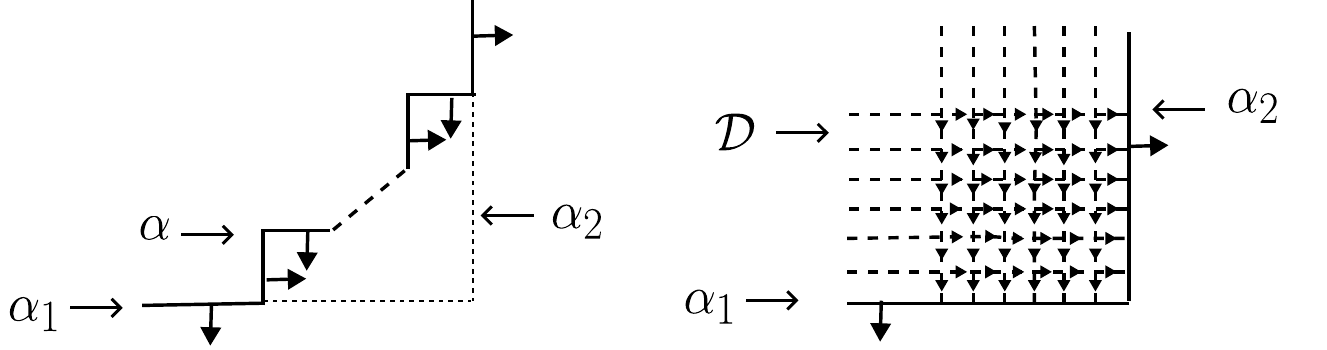}
  \caption{
  \small
$\alpha$ has finitely many corners and the current-determined region $\mathcal D$ is a quadrant region.  }
 \label{figorientedgeodesic2}
\end{minipage}
\end{figure}

\begin{figure}[htbp]
\centering
\begin{minipage}{\linewidth}
\centering
       \includegraphics[height=0.32\linewidth,width=1.14\linewidth]{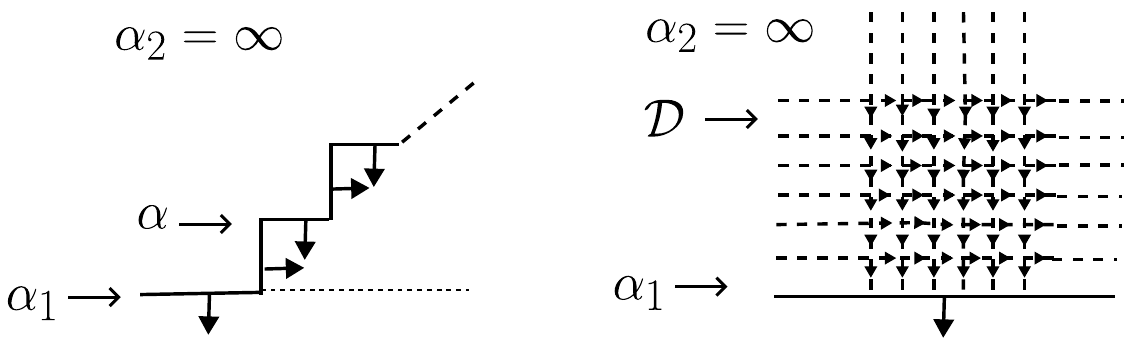}
  \caption{
  \small
$\alpha$ has infinitely many corners in one direction and the current-determined region $\mathcal D$ is a half plane.  }
 \label{figorientedgeodesic3}
\end{minipage}
\end{figure}

\begin{figure}[htbp]
\centering
\begin{minipage}{\linewidth}
\centering
       \includegraphics[height=0.32\linewidth,width=1.14\linewidth]{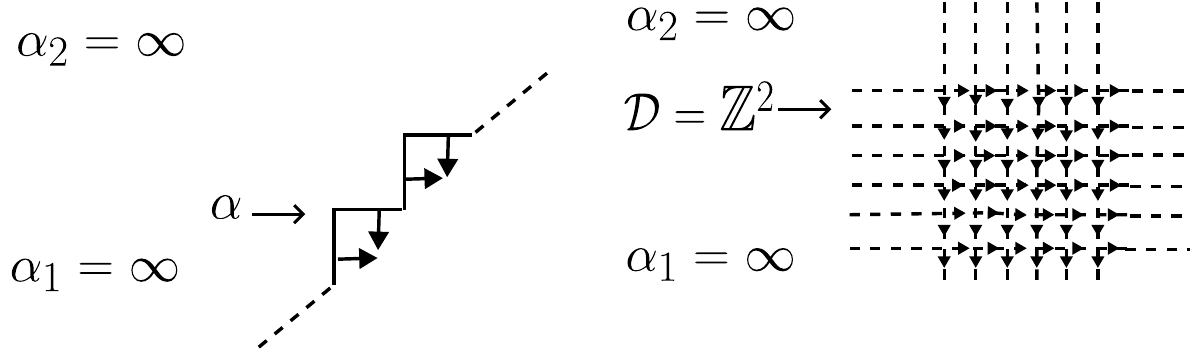}
  \caption{
  \small
$\alpha$ has bi-infinitely many corners and the current-determined region $\mathcal D$ is the whole plane or $\Z^2$.  }
 \label{figorientedgeodesic4}
\end{minipage}
\end{figure}
\lmd
\pf
The first statement can be deduced by the equation (\ref{oneLaplacian}) and Corollary \ref{minimalcurrent}. The currents in $\mathcal D$ yileds that $0\in \Delta_1(1_{\mathcal D})$. Using Corollary \ref{minimalcurrent} again, we have that $\m^\prime:=\m\cup\mathcal D$ is minimal. Applying Lemma \ref{orientedboundary}, one can obtain that $\alpha_1$ is an oriented boundary geodesic line (ray, line, resp.) in Fig. \ref{figorientedgeodesic1} (\ref{figorientedgeodesic2}, \ref{figorientedgeodesic3}, resp.), and $\alpha_2$ is an oriented boundary geodesic ray in Fig. \ref{figorientedgeodesic2}. Thus, we get that $\alpha^\prime:=\alpha_1$ ($\alpha_1\cup\alpha_2,\alpha_1,$ resp.) is an oriented boundary geodesic line in Fig. \ref{figorientedgeodesic1} (Fig. \ref{figorientedgeodesic2}, Fig. \ref{figorientedgeodesic3}, resp.).
\pfd
\co\label{twoboundary}
If the minimal subgraph $\m$ contains two oriented boundary geodesic lines $\alpha,\beta$ in different connected components of $\de\m$, then $\alpha,\beta$ are both of the form in Fig. \ref{figorientedgeodesic2} in Lemma \ref{orientedgeodesic}.
\cod
\pf
We argue by contradiction. Note that $\alpha\cap\beta=\emptyset$.

Case 1. $\alpha$ is of the form in Fig. \ref{figorientedgeodesic1} or Fig. \ref{figorientedgeodesic3}. By Lemma \ref{orientedgeodesic} we may assume $\alpha^\prime$ is a horizontal line.

If $\beta$ is of the form in Fig. \ref{figorientedgeodesic1} or Fig. \ref{figorientedgeodesic3}, and we assume that $\beta^\prime$ is a vertical line. Then it is easy to deduce that $\alpha\cap\beta\neq\emptyset$, which is impossible. If $\beta^\prime$ is also a horizontal line, then we can get a new minimal subgraph containing two boundary horizontal lines $\alpha^\prime,\beta^\prime$ by Lemma \ref{orientedgeodesic}. This is impossible by Lemma \ref{noboundedstrip}.

If $\beta$ is of the form in Fig. \ref{figorientedgeodesic2}, then we can get a new minimal subgraph containing two boundary horizontal rays by Lemma \ref{orientedgeodesic}. This is impossible by Lemma \ref{noboundedstrip}.

If $\beta$ is of the form in Fig. \ref{figorientedgeodesic4}, then it always holds that $\alpha\cap\beta\neq\emptyset$. This is impossible.

Case 2. $\alpha,\beta$ are both of the form in Fig. \ref{figorientedgeodesic4}. Since $\alpha\cap\beta=\emptyset$,
 $\alpha,\beta$ divide the plane into three domains $\mathcal{D}_1,\mathcal{D}_2,\mathcal{D}_3$, where $\mathcal{D}_2$ is bounded by $\alpha$ and $\beta$; see Fig. \ref{figorientedgeodesic5}.
\begin{figure}[htbp]
\centering
\begin{minipage}{\linewidth}
\centering
       \includegraphics[height=0.2\linewidth,width=0.6\linewidth]{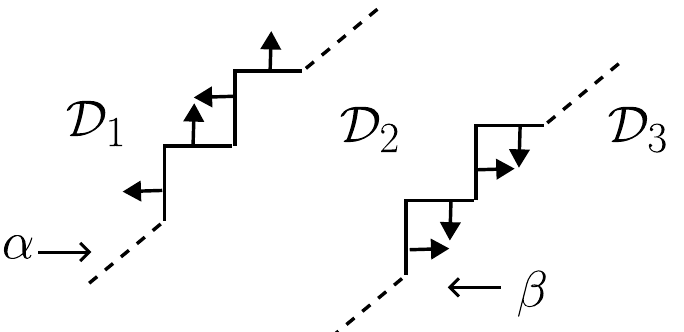}
  \caption{
  \small
  }
 \label{figorientedgeodesic5}
\end{minipage}
\end{figure}
Thus the currents determined by $\alpha$ are inconsistent with these determined by $\beta$ by Lemma \ref{orientedgeodesic}, i.e. $0\notin \Delta_1(1_{\m})$.

\pfd


\lm\label{rigidityforonecorner}
Assume $\de\m$ contains no loops. Let $\alpha$ be a boundary geodesic line in $\de\m$ with exactly one corner $x\in \alpha$. Then the connected component $\mathcal {D}$ of $\m$ containing $x$ in a quadrant region $\mathcal{E}$ enclosed by $\alpha$ is $\alpha$ or $\mathcal{E}$.
\lmd
\pf
If not, then one can find a vertex $p\in\mathcal{E}\backslash\mathcal{D}$. We may assume that $\alpha=\alpha_1\cup\alpha_2$, where $\alpha_1,\alpha_2$ are flat geodesic rays from $x$; see Fig. \ref {figmainthm4}. By Corollary \ref{convexclosed2}, two flat geodesic rays $\alpha_1^\prime,\alpha_2^\prime$ from $p,$ which don't intersect with $\mathcal{D}$; see Fig. \ref{figmainthm4}.
\begin{figure}[htbp]
\centering
\begin{minipage}{\linewidth}
\centering
       \includegraphics[height=0.25\linewidth,width=0.50\linewidth]{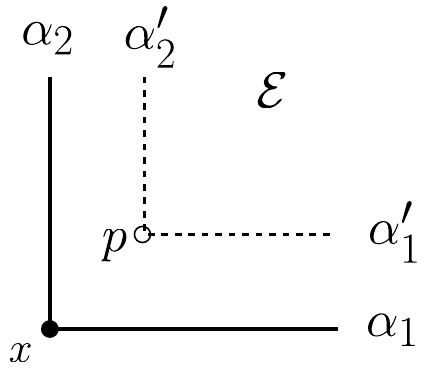}
  \caption{
  \small
}
 \label{figmainthm4}
\end{minipage}
\end{figure}
Then $\mathcal{D}\cap\mathcal{E}$ is bounded by $\alpha_1\cup\alpha_2$ and $\alpha_1^\prime\cup\alpha_2^\prime$.

Since $\mathcal{E}\backslash\mathcal{D}\neq\emptyset$, there is another bi-infinite simple boundary path $\beta\subset\de\m\cap\mathcal{D}$ with $\beta\neq\alpha_1\cup\alpha_2$. Since $\de\m$ has no loops, one can find an infinite simple subpath $\beta_1\subset\beta$ such that $\beta_1\cap(\alpha_1\cup\alpha_2)=\emptyset.$
Thus we can find two disjoint infinite simple boundary paths of $\mathcal{D}$ in the unbounded strip enclosed by $\alpha_1^\prime$ and $\alpha_1$ or enclosed by $\alpha_2^\prime$ and $\alpha_2$. This is impossible by Corollary \ref{noboundedstrip2}.

\pfd

\co\label{rigidityforfinitecorners}
Assume $\de\m$ contains no loops. Let $\alpha$ be a boundary geodesic line in $\de\m$ with finitely many corners. Then the connected component of $\m$ containing $\alpha$ in a quadrant region $\mathcal{D}$ enclosed by $\alpha$ is $\alpha$ or $\mathcal{D}$.
\cod
\pf
The argument of Lemma \ref{rigidityforonecorner} still works with minor modifications.
\pfd

\lm\label{nomoretwocomponents}
Assume that $\de\m$ is geodesic and contains no loops. Then both $\de\m$ and $\m$ have at most two connected components.
\lmd
\pf
If $\de\m$ has more than two connected components, then there are three oriented disjoint boundary geodesic lines by Lemma \ref{noisolatedpoint} and Lemma \ref{infinitybdy}. Using Corollary \ref{twoboundary}, one easily sees that there are two disjoint parallel boundary geodesic rays in a bounded strip. It is a contradiction by Lemma \ref{noboundedstrip}.

\pfd

Now we can prove Theorem \ref{main1}.

\pf[Proof of Theorem \ref{main1}]
We divide it into three cases.

Case 1. $\de\m$ is non-geodesic.

We may assume that $\de\m$ contains a path $x\sim z\cup[z,w]\cup w\sim y$ such that $[x,y]$ is not a boundary path with $d(x,y)=d(z,w)$, and $z,w$ are horizontal; see Fig. \ref{figmainthm1}. Then the rectangle determined by $x,y,w,z$ is contained in $\m$ by Corollary \ref{convexclosed2}.
\begin{figure}[htbp]
\centering
\begin{minipage}{0.8\linewidth}
\centering
       \includegraphics[height=0.15\linewidth,width=0.44\linewidth]{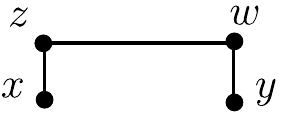}
  \caption{
  \small
}
 \label{figmainthm1}
\end{minipage}
\end{figure}
If $\deg(z)\geq 3$ and $\deg(w)\geq 3$, then this contradicts Lemma \ref{digonaldistrubution} by symmetry of $z$ and $w.$

If $\deg(z)=2$ and $\deg(w)=2$, then we get the new subgraph $\m_1$ by removing the rectangle $\mathcal R(x,y,z,w)$ except $[x,y]$. By Lemma \ref{orientedboundary}, we have $$|\partial\m\cap E_\om|=|\partial\m_1\cap E_\om|+2>|\partial\m_1\cap E_\om|,$$ where $\om:=\mathcal R(x,y,z,w)$. This contradicts the minimality of $\m.$

Thus, either $\deg(z)\geq 3,\deg(w)=2$ or $\deg(w)\geq 3,\deg(z)=2$ occurs. So that we can assume $\deg(w)\geq 3$ and $\deg(z)=2$.
Denote by $\mathcal C$ the connected component of $\m$ containing $w$ and $\mathcal C_1 (\mathcal C_2,resp.),$ the part of $\mathcal C$ in the domain $I(II,resp.)$; see Fig. \ref{figmainthm2}.
\begin{figure}[htbp]
\centering
\begin{minipage}{\linewidth}
\centering
       \includegraphics[height=0.32\linewidth,width=0.64\linewidth]{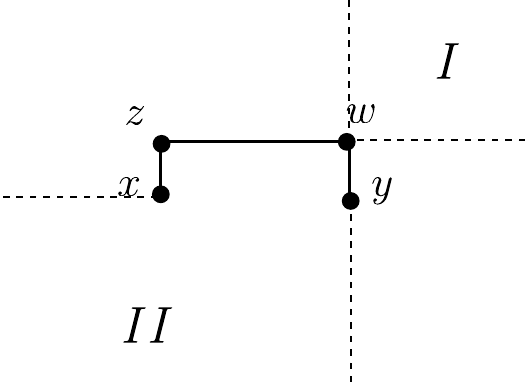}
  \caption{
  \small
}
 \label{figmainthm2}
\end{minipage}
\end{figure}

Step 1 of Case 1. $\mathcal C_1(\mathcal C_2,resp.)$ is enclosed by two geodesic rays from $w$ in the domain $I(II,resp.)$.

\begin{claim}\label{noloops}
$\de\mathcal C_1$ contains no loops.
\end{claim}
\pf[Proof of Claim \ref{noloops}]
Otherwise, $\de\mathcal C_1$ contains exactly one unit square $\mathcal R(\hat x,\hat y,\hat z,\hat w)$ by Corollary \ref{squareloop2}. We can assume that $\hat x,\hat y,\hat w,\hat z$ are arranged anticlockwise, $\hat x,\hat y$ are horizontal, and $\hat x$ is on the left of $\hat y$. By Lemma \ref{twocrossbd}, there is a boundary simple path $\alpha$ connecting $w,\hat x$ in the domain $I$. Applying Lemma \ref{orientedboundary} to the boundary simple paths $\alpha\cup \hat x\sim\hat y\sim \hat w$ and $\alpha\cup \hat x\sim\hat z\sim \hat w$, we have $\deg(\hat y)=\deg(\hat z)=2.$

Consider the new subgraph $\m_1$ by removing the part $\mathcal S$ in the rectangle determined by $x,y,z,\hat w, \hat y,\hat z$ from $\m$ except $[x,y],[\hat x,\hat y].$ Since $\deg(z)=\deg(\hat y)=\deg(\hat z)=2$, we have $$|\partial\m\cap E_\om|\geq|\partial\m_1\cap E_\om|+2\cdot 2-2>|\partial\m_1\cap E_\om|,$$ where $\om:=\mathcal S$. This contradicts the minimality of $\m.$
\pfd

Using Claim \ref{noloops} and Lemma \ref{nolongline}, we may assume that $\hat\beta_1:=\beta\cup\beta_1=(\hat\beta_2:=\beta\cup\beta_2,resp.)$ is the
leftmost(rightmost,resp.) simple infinite path in $\mathcal C_1$ such that $\beta=\hat\beta_1\cap\hat\beta_2$ is a path connecting $w$ and $\hat w$.
Let $l_1(l_2,resp.)$ be a vertical(horizontal,resp.) line through $w$; see Fig. \ref{figmainthm3}. Applying Corollary \ref{noboundedstrip2}, the distance projection of $\beta_1(\beta_2,resp.)$ on $l_1(l_2,resp.)$ is an infinite ray.
\begin{figure}[htbp]
\centering
\begin{minipage}{\linewidth}
\centering
       \includegraphics[height=0.42\linewidth,width=0.84\linewidth]{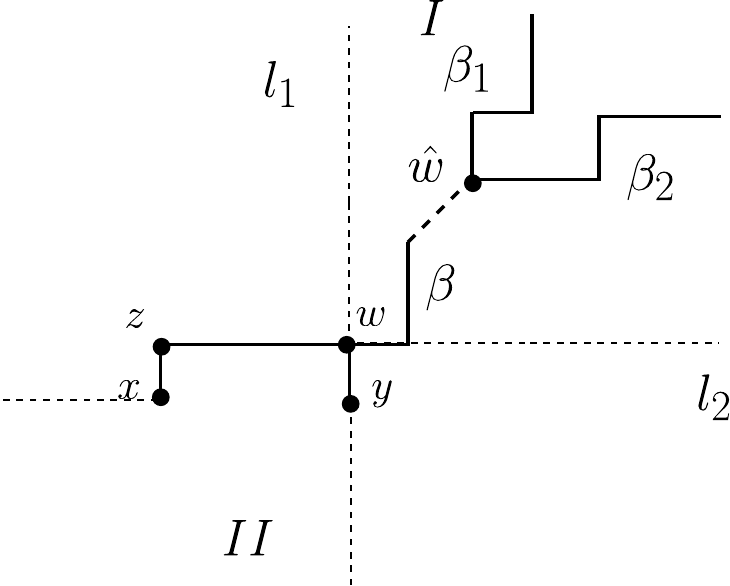}
  \caption{
  \small
}
 \label{figmainthm3}
\end{minipage}
\end{figure}

\begin{claim}\label{geodesicraybdy}
$\beta_1,\beta_2$ are geodesic, and the length of $\beta$ is at most one.
\end{claim}
\pf[Proof of Claim \ref{geodesicraybdy}]
If $\beta_1$ is not geodesic, then there are four cases for some non-geodesic boundary subpath $\hat w\sim\cdots\sim x_1\sim x_2\sim \cdots\sim x_3\sim x_4\subset\beta_1.$ If $x_1(x_4,\ resp.)$ is right-horizontal to $x_2(x_3,\ resp.)$, $x_4$ is up-vertical to $x_1$, then there is a boundary edge $(y_1,y_2)\subset \beta\cup\hat w\sim\cdots\sim x_1$ such that $x_1,x_4,y_1$ and $x_2,x_3,y_2$ are vertical. This contradicts Corollary \ref{twocrossbd}.

If $x_1(x_4,\ resp.)$ is down-vertical to $x_2(x_3,\ resp.)$, $x_1$ is left-horizontal to $x_4$, then by that fact that the distance projection of $\beta_1$ on $l_1$ is an infinite ray, there is an edge $(y_1,y_2)\subset \beta_1$ such that $x_1,x_4,y_1$ and $x_2,x_3,y_2$ are horizontal. This contradicts Corollary \ref{twocrossbd}.

If $x_1(x_4,\ resp.)$ is left-horizontal to $x_2(x_3,\ resp.)$, $x_4$ is up-vertical to $x_1$, then we have $\mathcal R(x_1,x_2,x_3,x_4)\subset \m$ by Corollary \ref{convexclosed2}. This contradicts that $\beta_1$ is leftmost.

If $x_1(x_4,\ resp.)$ is up-vertical to $x_2(x_3,\ resp.)$, $x_1$ is left-horizontal to $x_4$, then we have $\mathcal R(x_1,x_2,x_3,x_4)\subset \m$ by Corollary \ref{convexclosed2}. This contradicts that $\beta_1$ is leftmost.

Hence we have that $\beta_1$ is geodesic. The same argument yields that $\beta_2$ is geodesic.

Denote by $L(\beta)$ the length of $\beta$. Note that $\beta$ is geodesic by Corollary \ref{convexclosed2} and Claim \ref{noloops}. Now consider a new subgraph $\m_1$ obtained by removing the part of $\mathcal C$ in the rectangle $\mathcal R^\prime$ determined by $x,y,z,\hat w$, from $\m$ except $[x,y]$ and $\hat w$. Combining Lemma \ref{orientedboundary} with the minimality of $\m$, we have $$|\partial\m_1\cap E_\om|\geq|\partial\m\cap E_\om|\geq|\partial\m_1\cap E_\om|+2(L(\beta)-1),$$ where $\om:=\mathcal C\cap\mathcal R^\prime$. This implies $L(\beta)\leq 1.$
\pfd

\begin{claim}\label{georayenclose}
The domain $\mathcal D_1$ enclosed by $\beta_1,\beta_2$ in domain $I$ is contained in $\mathcal C_1.$ As a consequence,
we have $\mathcal C_1=\mathcal D_1\cup\beta.$
\end{claim}

\pf[Proof of Claim \ref{georayenclose}]
By same arguments in Lemma \ref{rigidityforonecorner}, one can show that $\mathcal{C}_1-\beta=\mathcal{D}_1\ or\ \mathcal{C}_1-\beta=\beta_1\cup\beta_2$. By definition, $\mathcal{C}_1-\beta=\beta_1\cup\beta_2$ implies that $\beta_1\cup\beta_2$ is an isolated path. It is impossible by Lemma \ref{nolongline}. Thus we prove the result.

\pfd

Similar arguments yield that $\mathcal C_2$ is the domain enclosed by two geodesic rays $\beta_1^\prime$ and $\beta_2^\prime$ from $w$, where the subpath $w\sim\cdots\sim z\sim x\subset\beta_1^\prime$, the subpath $w\sim y\subset\beta_2^\prime,$ and $\beta_1^\prime\cap\beta_2^\prime=\{w\}$.

Step 2 of Case 1. The classification of $\m$ of Case 1 is indeed as listed in Theorem \ref{main1} by Lemma \ref{orientedgeodesic}.
\begin{claim}\label{corner1}
At most one of $\beta_1$ and $\beta_2$ has corners. If $\beta_1$($\beta_2$,resp.) has corners, then $\beta_1=\hat{\beta_1}$ ($\beta_2=\hat{\beta_2}$,resp.) contains exactly two corners and the two corners are adjacent.
\end{claim}
\pf[Proof of Claim \ref{corner1}]
We argue via Corollary \ref{minimalcurrent} and the equation (\ref{oneLaplacian}). We give only some arguments and the remaining are similar. If both $\beta_1$ and $\beta_2$ have corners, then the current flowing out of $w\geq 3-1=2>0$; see the first picture in Fig. \ref{figmainthm5} for $L(\beta)=0$. This is a contradiction by Corollary \ref{minimalcurrent}. Note that if $\beta_1$ has corners, then it has at least two corners by Lemma \ref{noboundedstrip}. On the other hand, $\beta_1$ has at most two corners by Corollary \ref{minimalcurrent}; see Fig. \ref{figmainthm6}. The remaining claims are easy to verify by Corollary \ref{minimalcurrent}; see the last two pictures in Fig. \ref{figmainthm5} and Fig. \ref{figmainthm7}.
\begin{figure}[htbp]
\centering
\begin{minipage}{\linewidth}
\centering
       \includegraphics[height=0.3\linewidth,width=\linewidth]{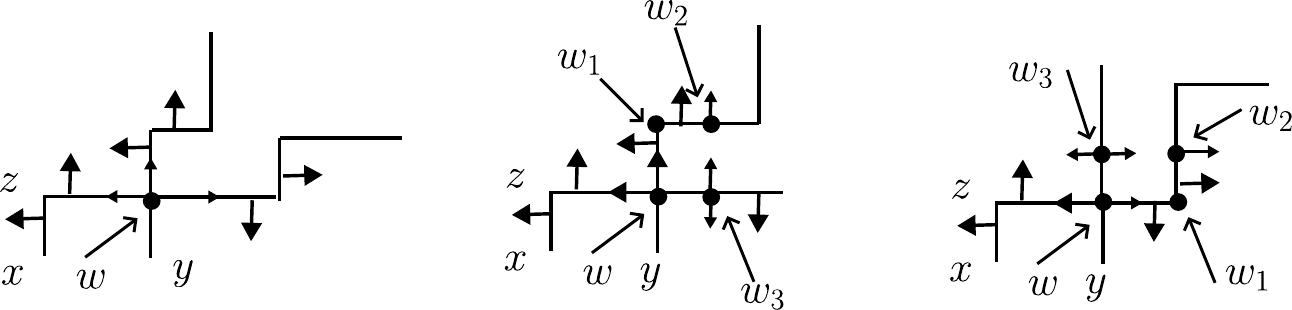}
  \caption{
  \small
The second picture shows $\Delta_1(1_{\m})(w)\neq 0$ or $\Delta_1(1_{\m})(w_3)\neq 0$ if two corners of $\beta_1$ are not adjacent.  }
 \label{figmainthm5}
\end{minipage}
\end{figure}

\begin{figure}[htbp]
\centering
\begin{minipage}{\linewidth}
\centering
       \includegraphics[height=0.23\linewidth,width=\linewidth]{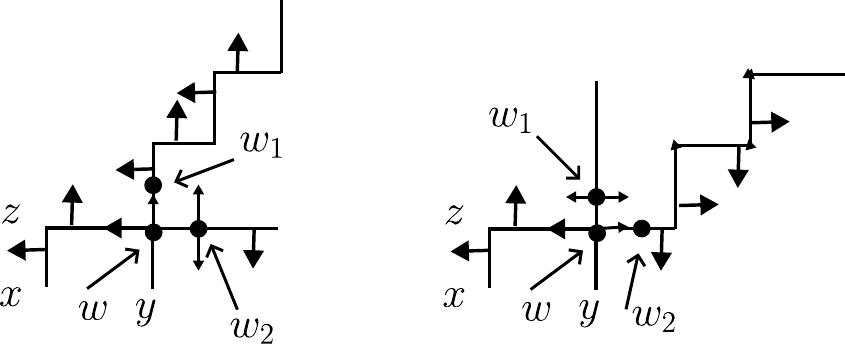}
  \caption{
  \small
The first picture shows $\Delta_1(1_{\m})(w)\neq 0$ or $\Delta_1(1_{\m})(w_2)\neq 0$ if $\beta_1$ has more than two corners.  }
 \label{figmainthm6}
\end{minipage}
\end{figure}

\begin{figure}[htbp]
\centering
\begin{minipage}{\linewidth}
\centering
       \includegraphics[height=0.23\linewidth,width=\linewidth]{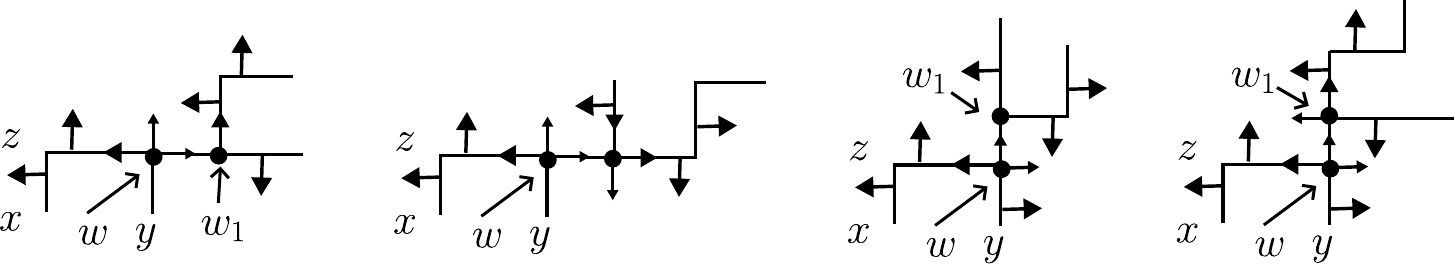}
  \caption{
  \small
 }
 \label{figmainthm7}
\end{minipage}
\end{figure}
\pfd

\begin{claim}\label{corner2}
For $\beta_1^\prime$ and $\beta_2^\prime$, denote by $\mathcal{R}_1,\mathcal{R}_2$ the rectangles containing $w$ and $p_1$, $w$ and $p_1$ respectively, and let $l_1$ and $h_1$, $l_2$ and $h_2$ represent sides of $\mathcal{R}_1,\mathcal{R}_2$ respectively.
Let $d_1,d_2$ be the length of $h_1,h_2$ respectively; see the graph $(a)$ in Fig. \ref{figmainthm8}. Then we have that $d_1\leq 2,d_2\leq 1.$ In particular,
$\beta_1^\prime$ ($\beta_2^\prime$,resp.) has at most four (two,resp.) corners.
If $\beta_2^\prime$ has two corners, then two corners in $\beta_1^\prime$ and $\beta_2^\prime$ are both adjacent.
\end{claim}
\pf[Proof of Claim \ref{corner2}]
\begin{figure}[htbp]
\centering
\begin{minipage}{\linewidth}
\centering
       \includegraphics[height=0.32\linewidth,width=\linewidth]{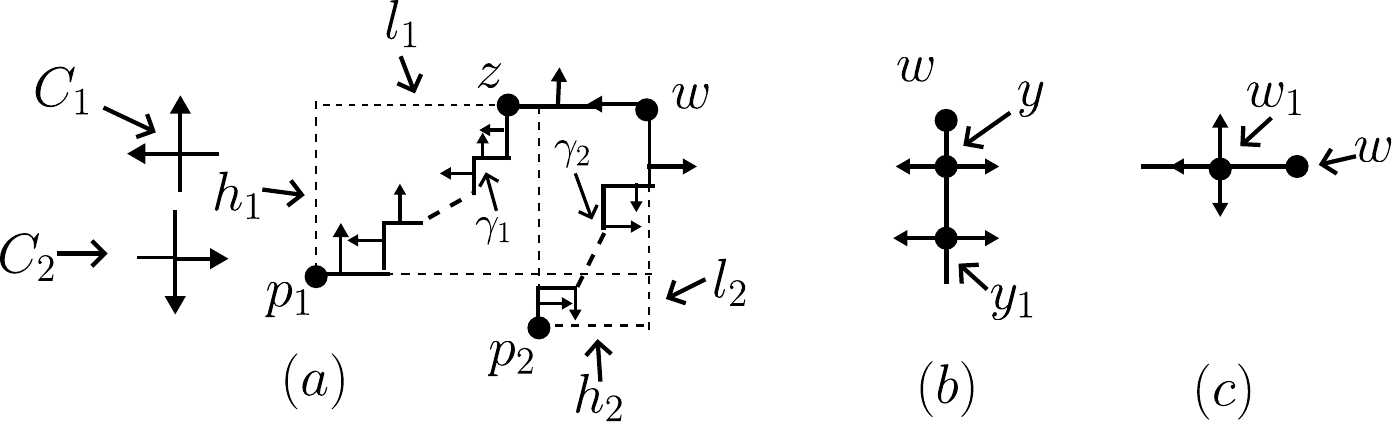}
  \caption{
  \small In $(a)$: The oriented paths $\gamma_1,\gamma_2$ from $w$ to $p_1,p_2$ are finite subpaths in $\beta_1^\prime,\beta_2^\prime$ respectively. $C_1,C_2$ denote the currents determined by the oriented paths $\gamma_1,\gamma_2$ in the interiors of the rectangles containing $w$ and $p_1$, $w$ and $p_2$ respectively. In $(b)$: $w\sim y\sim y_1$. In $(c)$: $w\sim w_1$ and $w_1\in [z,w]$.}
 \label{figmainthm8}
\end{minipage}
\end{figure}


If $d_1\geq 3$ , then this yields $(b)$ in Fig. \ref{figmainthm8} by the currents indicated by $C_1$ in Fig. \ref{figmainthm8} and Corollary \ref{minimalcurrent}. This implies that $\Delta_1(1_{\m})(y)\neq 0$ or $\Delta_1(1_{\m})(y_1)\neq 0$, which is impossible by Corollary \ref{minimalcurrent}. Similarly, if $d_2\geq 2$, then we have $(c)$ in Fig. \ref{figmainthm8} by the currents indicated by $C_2$ in Fig. \ref{figmainthm8}. This  implies that $\Delta_1(1_{\m})(w_1)\neq 0$, which is impossible by Corollary \ref{minimalcurrent}.
\pfd

\begin{claim}\label{nongeodesicconnected}
$\m$ is connected.
\end{claim}
\pf[Proof of Claim \ref{nongeodesicconnected}]
If not, let $\hat{\mathcal{C}}$ be another connected component of $\m$ and $\hat{\mathcal{C}}\cap\mathcal{C}=\emptyset.$ Then $\de\hat{\mathcal{C}}$ is geodesic. Otherwise, $\hat{\mathcal{C}}$ contains two quadrants located diagonally such that $\hat{\mathcal{C}}\cap\mathcal{C}\neq\emptyset$ by Claim \ref{corner1} and Claim \ref{corner2}, which is impossible.
So that one can find a geodesic line $\gamma$ in $\de\hat{\mathcal{C}}$ by Lemma \ref{noisolatedpoint} and Lemma \ref{infinitybdy}. Recalling Corollary \ref{twoboundary} for $\gamma$ and $\beta_1^\prime\cup\hat\beta_1$ or $\gamma$ and $\beta_2^\prime\cup\hat\beta_2$, we have two disjoint geodesic rays in bounded infinite strip, which contradicts Corollary \ref{noboundedstrip2}.
\pfd

Thus, $\m$ must be one of the forms in (\ref{main1-1}) of Theorem \ref{main1} by Claim \ref{geodesicraybdy}, Claim \ref{corner1}, Claim \ref{corner2} and Claim \ref{nongeodesicconnected}. On the other hand, one can use the argument of currents to show that these forms in (\ref{main1-1}) of Theorem \ref{main1} are indeed minimal by Corollary \ref{minimalcurrent}.

Case 2. $\de\m$ is geodesic and simple.

Subcase 2-1 The number of connected components of $\m$ is one, i.e. $\m$ is connected.

Subcase 2-1-1 If $\de\m$ is connected, then one easily deduces that $\m$ is one of the forms in Fig. \ref{fig3-1-1}-Fig. \ref{fig3-1-5} by Corollary \ref{minimalcurrent} and the condition that $\de\m$ is geodesic and simple.

Subcase 2-1-2 If $\de\m$ is not connected, then recalling the condition of Case 2, we obtain that $\de\m$ has exactly two geodesic and simple connected components $\beta_1,\beta_2$ by Lemma \ref{nomoretwocomponents}. These yield that $\beta_1,\beta_2$ are oriented and of the form in Fig. \ref{figorientedgeodesic2} by Corollary \ref{twoboundary}.

Subcase 2-1-2-1 If $\m$ is of the form in Fig. \ref{fig3-2-1}, then we have the following claim.
\begin{claim}\label{3-2-1}
In this case, $\m$ is minimal if and only if $d\leq h+2.$
\end{claim}
\pf[Proof of Claim \ref{3-2-1}]
\begin{figure}[htbp]
\centering
\begin{minipage}{\linewidth}
\centering
       \includegraphics[height=0.42\linewidth,width=0.650\linewidth]{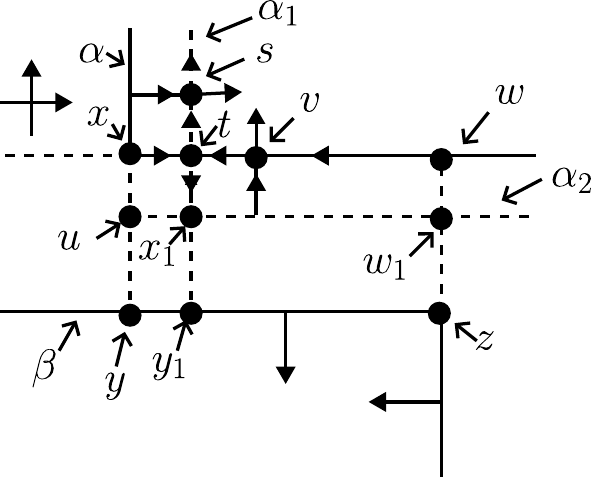}
  \caption{
  \small This connected subgraph is bounded by two geodesic lines $\alpha,\beta$. $h=d(x,y)=d(z,w),d=d(x,w)=d(y,z)$. $R(x,u,x_1,t)$ is an unit square. $s\sim t$. $\alpha_1\supset[x_1,t]$ ($\alpha_2\supset[x_1,w_1]$, resp.) denotes the vertical (horizontal,resp.) ray from $x_1$.
}
 \label{figmainthm9}
\end{minipage}
\end{figure}
Since $\m$ is of the form in Fig. \ref{fig3-2-1}, $\m$ is minimal if and only if the new subgraph $\m_1$ in Fig. \ref{figmainthm9} is also minimal by Lemma \ref{orientedgeodesic}. If we remove the rectangle $R(x,y,z,w)$ to get another new graph $\m_1^\prime$, then we have $$|\partial\m_1^\prime\cap E_\om|=|\partial\m_1\cap E_\om|+2(h+1)-2(d-1)\geq|\partial\m_1\cap E_\om|,$$ where $\om:=R(x,y,z,w)$. This implies that $d\leq h+2.$

On the other hand, we shall prove that $\m_1$ is minimal if $d\leq h+2.$ We argue it by induction on $h$. If $h=0$, then we get the result by Corollary \ref{minimalcurrent}; see Fig. \ref{figmainthm10}. To be precise,
\begin{figure}[htbp]
\centering
\begin{minipage}{\linewidth}
\centering
       \includegraphics[height=0.36\linewidth,width=\linewidth]{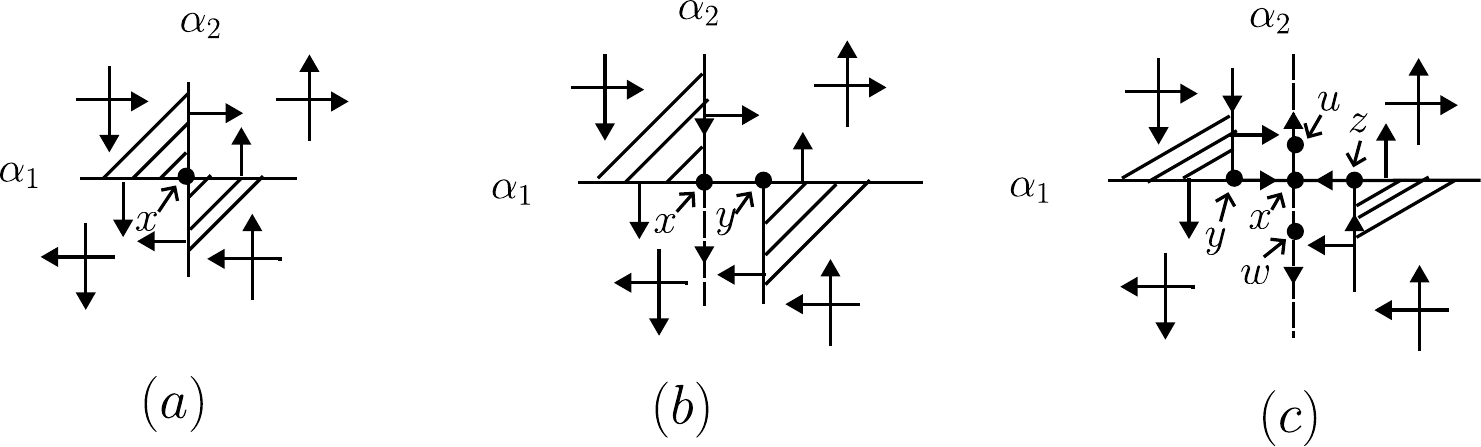}
  \caption{\small $\alpha_1$($\alpha_2$,resp.) are horizontal(vertical,resp.) lines through $x$ and they divide the plane into four quadrant regions. The crossed arrows denote the currents in the interiors of respective quadrant regions.
}
 \label{figmainthm10}
\end{minipage}
\end{figure}
for $d=0$, the currents of $\alpha_1,\alpha_2$ can be arbitrarily given (for example, zero currents) whenever they are of sum zero along respective lines; see $(a)$ in Fig. \ref{figmainthm10}. For $d=1,$ the currents of $\alpha_1$ can be arbitrarily given whenever they are of sum zero along respective lines. The currents of $\alpha_2$ are indicated in Fig. \ref{figmainthm10}; see $(b)$ in Fig. \ref{figmainthm10}. For $d=2,$ one can see $(c)$ in Fig. \ref{figmainthm10} to find that the minimal currents of $\alpha_1,\alpha_2$ are indicated.

Now let the current from $t$ to $x_1$ be as in Fig. \ref{figmainthm9}. Then we obtain the currents in Fig. \ref{figmainthm9} by Corollary \ref{minimalcurrent} and the currents in the interior of the upper half plane bounded by the horizontal line through $x$ and $t$, which are indicated by crossed arrows in Fig. \ref{figmainthm9}. Denote by $\alpha_1^\prime$ the vertical ray from $t$ in $\alpha_1$, and $\alpha_2^\prime$ the horizontal ray from $t$ through $w$.  Then one gets a new connected graph $\m_2$ bounded by $\alpha_1\cup\alpha_2$ and $\beta$; see Fig. \ref{figmainthm9}. Note that $d(x_1,y_1)=h-1,d(x_1,w_1)=d-1$, and one can reverse the currents of the geodesic line $\alpha_1^\prime\cup\alpha_2^\prime$. One can show that this preserves that $0\in\Delta_1(1_{\m_2})$. Thus, by applying induction assumption we finish the proof.

\pfd
Subcase 2-1-2-2 If $\m$ is of the form in Fig. \ref{fig3-2-2}, then $\m$ is minimal if and only if the new subgraph $\m_1$ in Fig. \ref{figmainthm11} is minimal by Lemma \ref{orientedgeodesic}. Now we show that $\m_1$ in Fig. \ref{figmainthm11} is indeed minimal by Corollary \ref{minimalcurrent}. More precisely, $\alpha,\beta$ divide the plane into three regions $\mathcal{D}_1,\mathcal{D}_2$ and $\m_1$; see Fig. \ref{figmainthm11}. The currents in the interiors of $\mathcal{D}_1,\mathcal{D}_2$ are indicated as the crossed arrows in Fig. \ref{figmainthm11}, and the currents for the horizontal lines or vertical lines through the rectangle $R(x,y,z,w)$ in Fig. \ref{figmainthm11} only need to be of sum zero along these lines. Hence, this yields that $0\in\Delta_1(1_{\m_1})$.
\begin{figure}[htbp]
\centering
\begin{minipage}{0.8\linewidth}
\centering
       \includegraphics[height=0.42\linewidth,width=0.750\linewidth]{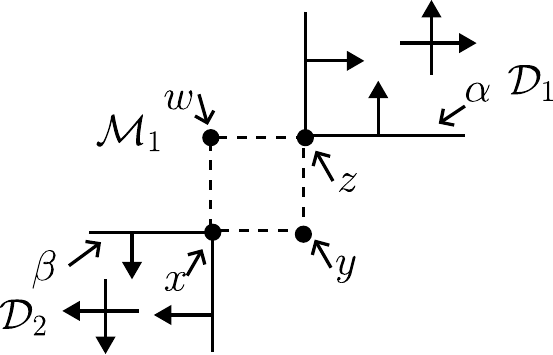}
  \caption{\small $\m_1$ is a connected subgraph bounded by two geodesic lines $\alpha$ and $\beta$.
}
 \label{figmainthm11}
\end{minipage}
\end{figure}

Subcase 2-2 The number of connected components of $\m$ is two.

One can show that $\de\m$ has exactly two connected components by Lemma \ref{nomoretwocomponents}. Recalling that $\de\m$ is geodesic and simple, one obtains that the complementary subgraph $\m^c$ of $\m$ is connected. Lemma \ref{complementaryminimal} yields that $\m^c$ is also minimal.

Case 3. $\de\m$ is geodesic and non-simple.

Subcase 3-1  $\de\m$ contains some loops. Then one of connected components of $\m$ contains exactly one unit square by Corollary \ref{squareloop2}.


We may assume that there is a square $R(x,y,z,w)\subset\de\m\cap\mathcal{C},$ where $\mathcal{C}$ is the connected component of $\m$ containing $x$. Using Lemma \ref{digonaldistrubution} and symmetry of $x,z$ and $y,w$, one can assume that $\mathcal{C}$ is in the following shadow region; see Fig. \ref{figmainthm12}.
\begin{figure}[htbp]
\centering
\begin{minipage}{\linewidth}
\centering
       \includegraphics[height=0.30\linewidth,width=0.40\linewidth]{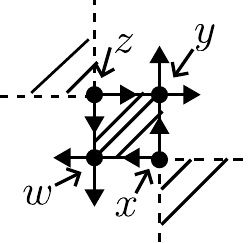}
  \caption{\small
}
 \label{figmainthm12}
\end{minipage}
\end{figure}

One easily gets that the minimal currents for the square $R(x,y,z,w)$; see Fig. \ref{figmainthm12}. Then by $\Delta_1(1_{\m})(x)=\Delta_1(1_{\m})(z)=0,$ we deduce that there are two neighbors $z_1,z_2$ (distinct from $w,y$) of $z$ and two neighbors $x_1,x_2$ (distinct from $w,y$) of $x$.
Applying Lemma \ref{noisolatedpoint} and Corollary \ref{squareloop2} and recalling that $\de\m$ is geodesic, we have two flat boundary geodesic rays $\alpha_1,\alpha_2$ from $z$ through $z_1,z_2$ respectively and two flat boundary geodesic rays $\alpha_3,\alpha_4$ from $x$ through $x_1,x_2$ respectively. Using Lemma \ref{rigidityforonecorner}, it implies that $\mathcal{C}$ is isometric to the subgraph in Fig. \ref{fig2-2}. Similar arguments as in the proof of Claim \ref{nongeodesicconnected} yield that $\m=\mathcal{C}$.


Subcase 3-2 $\de\m$ contains no loops.

Denote by $\deg(p,\de\m)$ the number of neighbors of $p$ in $\de\m$ for any $p\in\m$, and by $\mathcal{C}$ the connected component of $\m$ containing $x$ for some $x\in\de\m$. Recalling that $\de\m$ is non-simple, we may assume that $\deg(x,\de\m)\geq 3.$

Subcase 3-2-1 Assume $\deg(x,\de\m)=4$ and $y,z,v,w$ are neighbors of $x$.

By Lemma \ref{noisolatedpoint}, Lemma \ref{infinitybdy} and the fact that $\de\m$ is geodesic, there are four boundary geodesic rays $\alpha_1,\alpha_2,\alpha_3,\alpha_4$ from $x$ through $y,z,v,w$ respectively. By Corollary \ref{convexclosed2} and recalling that $\de\m$ is geodesic and contains no loops, we obtain that these geodesic rays are all flat or have no corners. The four geodesic rays divide the plane into four quadrant regions. Applying Lemma \ref{rigidityforonecorner} to these four quadrant regions, we have that $\mathcal{C}$ is isometric to the subgraph with $h=0$ in Fig. \ref{fig2-1}. By similar arguments to Claim \ref{nongeodesicconnected}, we get $\mathcal{C}=\m$.

Subcase 3-2-2 Assume $\deg(x,\de\m)=3$ and $\deg(p,\de\m)\leq 3$ for any $p\in\de\m\cap\mathcal{C}$.
Let $y,z,w$ be neighbors of $x$ in $\de\m$, and $y$($z,w$, resp.) be up-vertical(left-horizontal, right-horizontal, resp.) to $x$.

Subcase 3-2-2-1 Assume $\deg(q,\de\m)=3$ for some $(x\neq) q\in\de\m\cap\mathcal{C}$.

We may assume $q\in\alpha_3$ with boundary neighbors $p,s,t$, and $p$ is in the geodesic subpath $\hat\alpha$ between $x$ and $q$. By similar arguments of Subcase 3-2-1, we get two flat boundary geodesic rays $\alpha_1,\alpha_2$ from $x$ through $y,z$ respectively, two flat boundary geodesic rays $\alpha_3,\alpha_4$ from $q$ through $s,t$ respectively, and $s$($t$,resp.) is right-horizontal(down-vertical,resp.) to $q$. Using Corollary \ref{rigidityforfinitecorners} for geodesic boundary line $\alpha_1\cup\alpha_2,\alpha_3\cup\alpha_4,\alpha_1\cup\hat\alpha\cup\alpha_3,\alpha_2\cup\hat\alpha\cup\alpha_4$, we have $\mathcal{C}=\mathcal{C}_1\cup\hat\alpha\cup\mathcal{C}_2,$ where $\mathcal{C}_1$ is the quadrant domain enclosed by $\alpha_1\cup\alpha_2$ and $\mathcal{C}_2$ is the quadrant domain enclosed by $\alpha_3\cup\alpha_4$. This yields that $\hat\alpha$ is of length at most two by Lemma \ref{nolongline}. By similar arguments in Claim \ref{nongeodesicconnected}, we have $\m=\mathcal{C}$ and it is isometric to the subgraph in Fig. \ref{fig2-1} with $1\leq h\leq 2$, or Fig. \ref{fig2-3}.

Subcase 3-2-2-2 Assume $\deg(q,\de\m)=2$ for any $x\neq q\in\de\m$.

By similar arguments in Subcase 3-2-1, we get three boundary geodesic rays $\alpha_1,\alpha_2,\alpha_3$ from $x$ through $y,z,w$ respectively and $\alpha_1$ is flat or vertical. If $\alpha_3$ has corners, then $\alpha_2$ is horizontal by Corollary \ref{convexclosed2} and the facts that $\de\m$ is geodesic and contains no loops. So that we can assume $\alpha_2$ is horizontal by symmetry of $\alpha_2,\alpha_3$. Let $\mathcal{D}_1,\mathcal{D}_2,\mathcal{D}_3$ be regions enclosed by $\alpha_1\cup\alpha_2,\alpha_1\cup\alpha_3,\alpha_2\cup\alpha_3$. We get $\mathcal{C}\cap\mathcal{D}_1=\mathcal{D}_1$ or $\mathcal{C}\cap\mathcal{D}_1=\alpha_1\cup\alpha_2$ by Lemma \ref{rigidityforonecorner}.

We claim $\mathcal{C}\cap\mathcal{D}_3=\mathcal{D}_3$ or $\mathcal{C}\cap\mathcal{D}_3=\alpha_2\cup\alpha_3$. Otherwise, recalling $\deg(q,\de\m)=2$ for any $x\neq q\in\de\m$ and $\de\m$ contains no loops, we have another boundary geodesic line $\beta$ with $\beta\cap(\alpha_1\cup\alpha_2\cup\alpha_3)=\emptyset$. By Lemma \ref{noboundedstrip}, there is a vertical edge $(a,b)$ in $\beta$ such that $a$ is up-vertical to $b$. Then one can find one geodesic ray $\beta_1$ in $\beta$ from $a,$ which doesn't contain $b,$ is above the horizontal line $\hat l$ through $b$. Then we get a contradiction by Corollary \ref{noboundedstrip2} for two disjoint geodesic rays $\alpha_2$ and $\beta_1$, which proves the claim.

Similarly, we have $\mathcal{C}\cap\mathcal{D}_2=\mathcal{D}_2$ or $\mathcal{C}\cap\mathcal{D}_3=\alpha_1\cup\alpha_3$. Note that $\alpha_1,\alpha_2,\alpha_3\subset\de\m.$ Thus one of them is isolated and it contradicts Lemma \ref{nolongline}.

Finally, we prove the result.
\pfd

\section{Geometry of high-dimensional minimal subgraphs}
For $x\in \R^n$ and $r>0,$
denote by $\hat S_r(x)=\{y\in \R^n:\max\limits_{1\leq i\leq n}|y_i-x_i|=r\}$ the $\infty$-normed sphere centered at $x$ of radius $r$. We say $\mathcal I^k\subset\R^n$ is a unit $k$-\emph{cube} if it is a $k$ dimensional cube with all edges of length one.
Similar to the setting for $\Z^2$, we identify the set $\{v_1,v_2,\cdots,v_{2^k}\}$ with some $k$-cube $\mathcal I^k$ if $\{v_1,v_2,\cdots,v_{2^k}\}$ are exactly the vertices of $\mathcal I^k$, where $k=0,1,\cdots,n$. So that any subgraph $\g$ in $\Z^n$ can be identified with a $n-$dimensional cell complex or a geometric space in $\R^n$, called the \emph{geometric realization} of $\g,$ which includes all cells whose vertices are contained in $\g.$ Given any subgraph $\g\subset\R^n$, let $\g^k$ be the $k$-\emph{skeleton} of $\g$(i.e. the union of $k$-cubes of $\g$) for $k=0,1,\cdots,n$.

\prp
For any minimal subgraph $\m$ in $\Z^n$, every connected component of $\m$ is also minimal.
\prpd
\pf
Assume $\m_1$ is one of connected components of $\m$. One can restrict the minimal recurrents of $\m$ to $\m_1$, which yields the result by Corollary \ref{minimalcurrent}.
\pfd

\lm\label{minideg}
Given any minimal subgraph $\m\subset\Z^n$ and any $x\in\m$, then $\deg(x)\geq n$. Moreover, it never occurs that $\deg(x)=\deg(y)=n$ for any $y\sim x$ in $\m$.
\lmd
\pf
If $\deg(x)<n$, then the currents flowing out of $x$ are at least $2n-\deg(x)>n$, but the currents flowing into $x$ are at most $\deg(x)<n.$ This is impossible by Lemma \ref{minimalcurrent}. 

If $\deg(x)=\deg(y)=n$, then applying Lemma \ref{minimalcurrent} to $x,$ one can get that the current from $y$ to $x$ is one. By symmetry of $x$ and $y$, one get the current from $x$ to $y$ is one. This is impossible.
\pfd

\lm\label{infinitybdy}
If $\m\subset\Z^n$ is minimal, then both $\m$ and $\de\m$ are infinite.
\lmd
\pf
If not, then $\de\m\subset B_r$ with some $r\in\N$. Hence $\m\subset B_r$ or $\m^c\subset B_r$. One can get a new graph $\m^\prime$ by removing $B_r$ from $\m$ or adding $B_r$ to $\m$. Comparing $\m$ with $\m^\prime$, one easily sees that $\m$ is not minimal.
\pfd

\prp\label{uppergrowth}
If $\m\subset\Z^n$ is minimal, then we have $|\de \m\cap\hat B_r(x)|\leq c(n)r^{n-1}$, where $c(n)$ is a constant depending only on $n$.
\prpd
\pf
One can get a new graph $\m^\prime$ by removing $\hat B_r(x)$ from $\m$ for any $x\in \m$ and $r\in\N$. Comparing $|\partial\m\cap E_{\hat B_r(x)}|$ with $|\partial\m^\prime\cap E_{\hat B_r(x)}|$ and using the minimality of $\m$, one can get
\begin{align*}
c(n)|\de \m\cap\hat B_r(x)|&\leq|\partial\m\cap E_{\hat B_r(x)}|
\\&\leq|\partial\m^\prime\cap E_{\hat B_r(x)}|
\\&\leq 2n|\de \hat B_r(x)\cap\m|
\\&\leq 4n^2(2r+1)^{n-1}.
\end{align*}

Thus we complete the proof.
\pfd

Now we are ready to prove Theorem \ref{redution3dim}.
\pf[Proof of Theorem \ref{redution3dim}]
It suffices to show the restriction of some minimal currents $I$ of $\m$ to $\m^3\cup\m^2,$ denoted by $\hat I,$ are well defined, and then they are minimal currents, i.e. the sum of the restrictive currents $\hat I$ for any vertex of $\Z^n$ is zero; see Lemma \ref{minimalcurrent}. Note that $\m^1$ may be nonempty. So that it suffices to show the currents $I$ on $\m^3\cup\m^2\cap\m^1$ agree with those induced naturally by $\m^3\cup\m^2$. Precisely speaking, the minimal current on any $(x,y)$ agrees with that induced by $\m^3\cup\m^2$, where $x\sim y$ for $x\in\m^3\cup\m^2$ and $y\in\m-\m^3\cup\m^2$. This yields that $\deg(x)\geq 3,\deg(y)\geq 3$ by Lemma \ref{minideg}.

If $\deg(y)\geq 5$, then $y\in\m^3\cup\m^2.$ If $\deg(y)=3$, then the result is true.

Thus we only need to consider the case $\deg(y)=4.$ Note that if $\deg(x)\geq 5$, one easily deduces that there exists a square in $\m$ containing the edge $(x,y)$. Hence $y\in\m^3\cup\m^2.$ We have the following two cases.

Case 1. $\deg(x)=3,\deg(y)=4$.

Subcase 1-1. The neighbors of $x$ and $y$ in $\m$ are as in Fig. \ref{3dim3-4-1}. Using Lemma \ref{minimalcurrent} to $x$ and $y$, we have that the currents on the oriented edges $(y,x),(p,y),(u,y),(v,y)$ are all one, indicated by arrows in Fig. \ref{3dim3-4-1}. Applying Lemma \ref{minimalcurrent} to $p$, we obtain that the current on $(w,z)$ is one, and hence $w\in\m$. By symmetry, we get $s,t\in\m$. Then the currents on the oriented edges $(t,z),(y,z),(t,z),(w,z)$ are all one and the sum of the currents flowing into $z\geq 4-2=2>0$. This is a contradiction.

Subcase 1-2. The neighbors of $x$ and $y$ in $\m$ are as in Fig. \ref{3dim3-4-2}. Using Lemma \ref{minimalcurrent} to $x$ and $y$, we have that the currents on the oriented edges $(y,x),(z,y)$ are both one. Since $y\in\m^1$, $u,v,w\notin\m$ and the sum of the currents flowing out of $z\geq 4-2=2>0$. This is a contradiction.

Subcase 1-3. The neighbors of $x$ and $y$ in $\m$ are as in Fig. \ref{3dim3-4-3}. Using Lemma \ref{minimalcurrent} to $x$ and $y$, we have that the currents on the oriented edges $(y,x),(u,y),(v,y)$ are all one, indicated by arrows in Fig. \ref{3dim3-4-3}. Using Lemma \ref{minimalcurrent} for $u,v$, one can get the currents on the oriented edges $(s,u),(t,v)$ are both one, and hence $s,t\in\m.$ Then the currents on the oriented edges $(w,p),(y,p),(s,p),(t,p)$ are all one. Thus the sum of the currents flowing in of $p\geq 4-2=2>0$. This is a contradiction.

Subcase 1-4. The neighbors of $x$ and $y$ in $\m$ are as in Fig. \ref{3dim3-4-4} or Fig. \ref{3dim3-4-5}. Using Lemma \ref{minimalcurrent} to $x$ and $y$, we deduce that the currents on oriented edges must be indicated by arrows in Fig. \ref{3dim3-4-4} or Fig. \ref{3dim3-4-5}. Noting that $y\notin\m^3\cup\m^2$, we have that the sum of currents flowing out of $z\geq 4-2=2>0$. This is a contradiction.

\begin{figure}
\begin{minipage}{0.49\linewidth}
\centering
       \includegraphics[height=0.6\linewidth,width=0.5\linewidth]{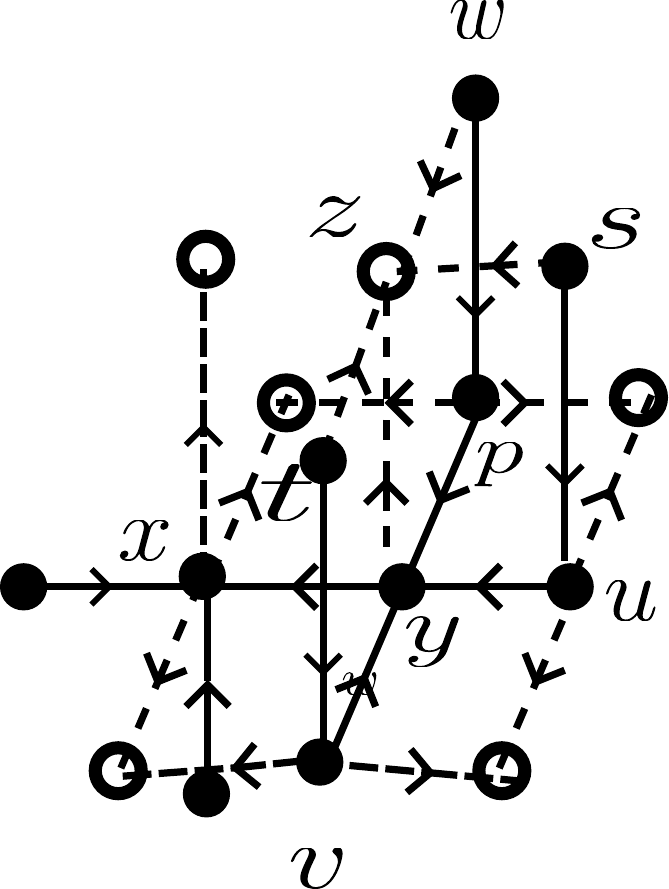}
  \caption{
  \small
 }
 \label{3dim3-4-1}
\end{minipage}
\begin{minipage}{0.49\linewidth}
\centering
       \includegraphics[height=0.6\linewidth,width=0.5\linewidth]{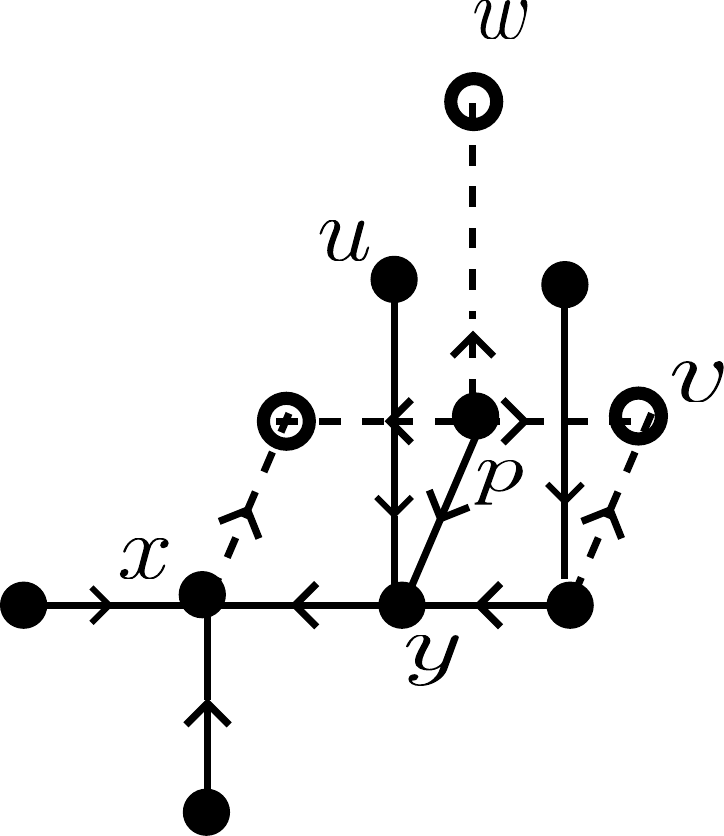}
  \caption{
  \small
}
 \label{3dim3-4-2}
\end{minipage}
\end{figure}

\begin{figure}
\begin{minipage}{0.49\linewidth}
\centering
       \includegraphics[height=0.6\linewidth,width=0.5\linewidth]{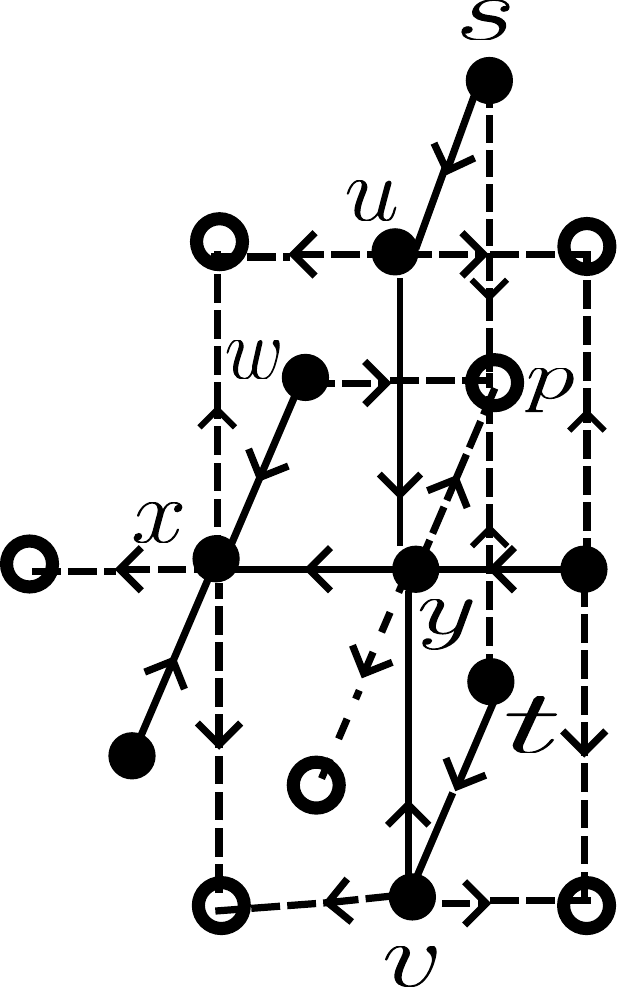}
  \caption{
  \small
 }
 \label{3dim3-4-3}
\end{minipage}
\end{figure}

\begin{figure}
\centering
\begin{minipage}{0.49\linewidth}
\centering
       \includegraphics[height=0.6\linewidth,width=0.5\linewidth]{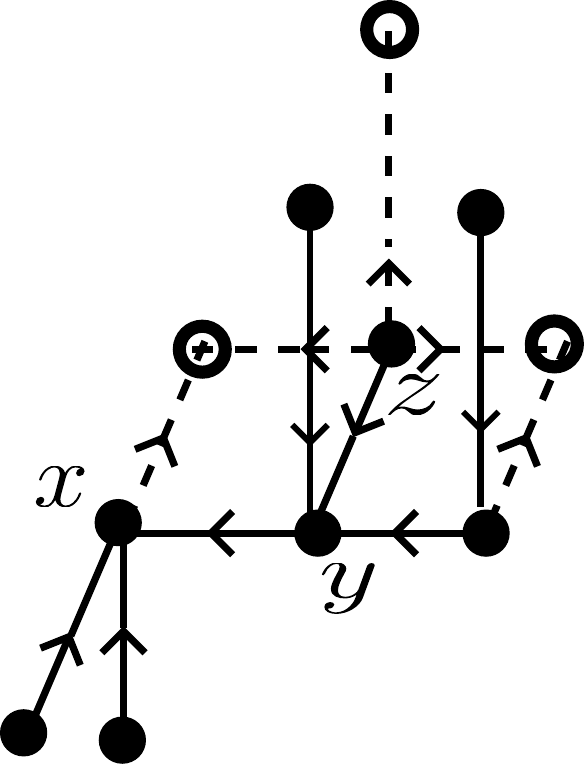}
  \caption{
  \small
 }
 \label{3dim3-4-4}
\end{minipage}
\centering
\begin{minipage}{0.49\linewidth}
\centering
       \includegraphics[height=0.6\linewidth,width=0.5\linewidth]{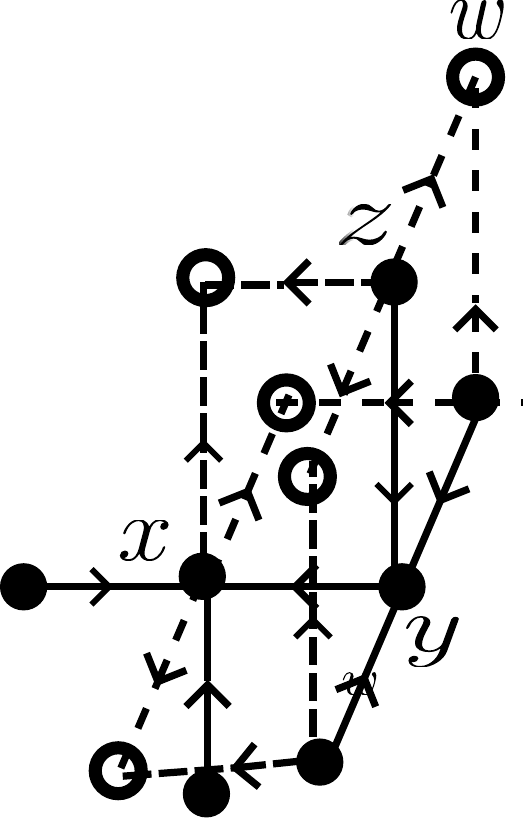}
  \caption{
  \small
}
 \label{3dim3-4-5}
\end{minipage}

\end{figure}

Case 2. $\deg(x)=\deg(y)=4$.

\begin{figure}
\centering
\begin{minipage}{0.49\linewidth}
\centering
       \includegraphics[height=0.6\linewidth,width=0.4\linewidth]{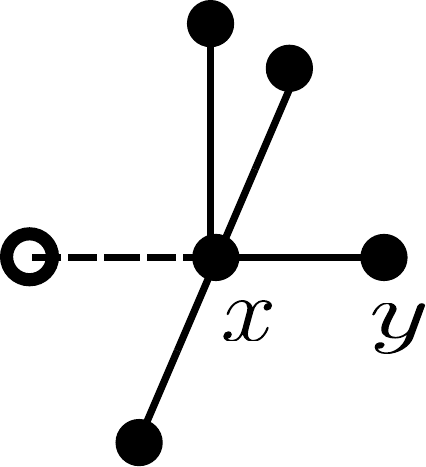}
  \caption{
  \small
 }
 \label{4dim3-4-1}
\end{minipage}
\centering
\begin{minipage}{0.49\linewidth}
\centering
       \includegraphics[height=0.6\linewidth,width=0.5\linewidth]{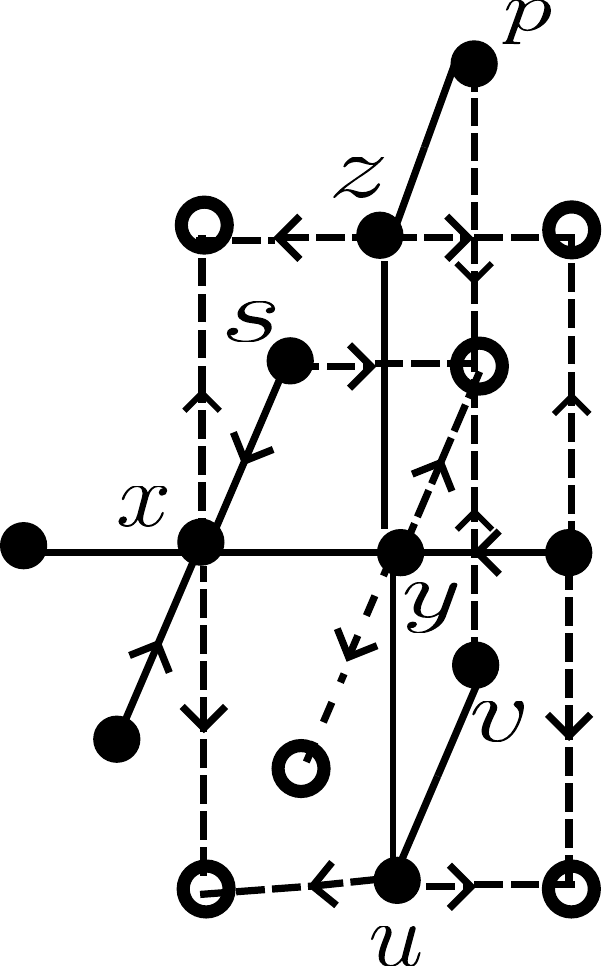}
  \caption{
  \small
}
 \label{4dim3-4-2}
\end{minipage}
\centering
\begin{minipage}{0.49\linewidth}
\centering
       \includegraphics[height=0.6\linewidth,width=0.5\linewidth]{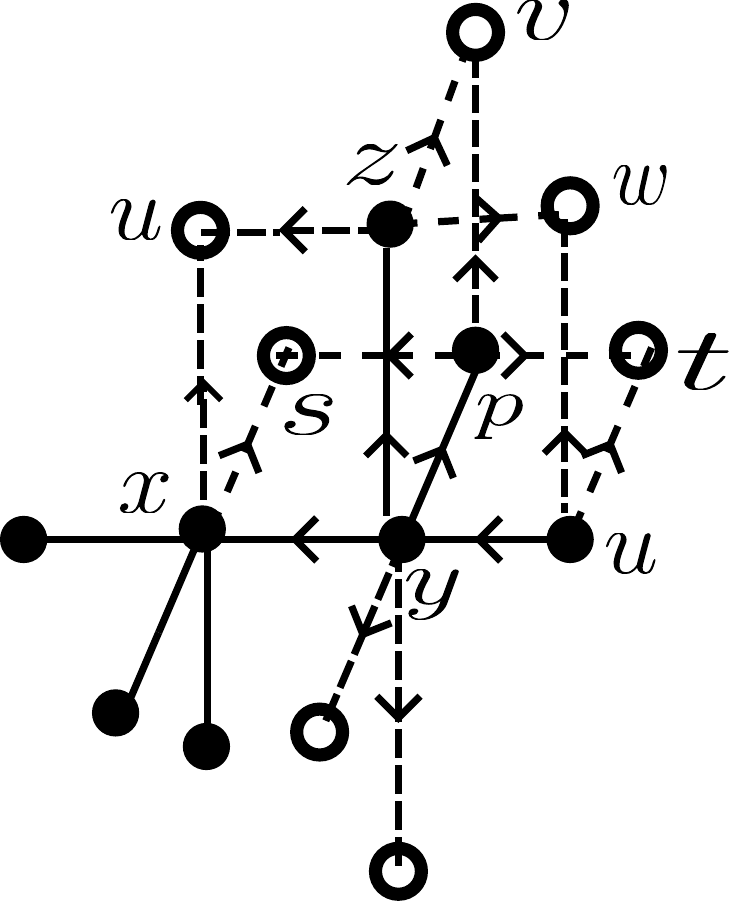}
  \caption{
  \small
  }
 \label{4dim3-4-3}
\end{minipage}



\end{figure}

Subcase 2-1. The neighbors of $x$ in $\m$ are as in Fig. \ref{4dim3-4-1}. Then it is obvious that $y\in\m^3\cup\m^2,$ which is impossible.

Subcase 2-2. The neighbors of $x$ and $y$ in $\m$ are as in Fig. \ref{4dim3-4-2}. If $p\notin\m$, then using Lemma \ref{minimalcurrent} to $z$, we have that the current on the oriented edge $(y,z)$ is one. So that we get the current on the oriented edge $(x,y)$ is one by applying Lemma \ref{minimalcurrent} to $y$, which agrees with that induced by $\m^3\cup\m^2$, and we finish the proof in this setting. By symmetry of $z,w,u$, the remaining setting in this subcase is that $p,q,v\in\m$. In this case, by direct computation the currents flowing into $t\geq 4-2=2>0$, which is impossible.

Subcase 2-3. The neighbors of $x$ and $y$ in $\m$ are as in Fig. \ref{4dim3-4-3}. It follows that $u,v,w\notin\m$ by $y\in\m^1$. By Lemma \ref{minimalcurrent} for $z$, we have that the current on the oriented edge $(y,z)$ is one. Recalling that $\deg(y)=4$ and using Lemma \ref{minimalcurrent} to $y$, we deduce that the current on the oriented edge $(x,y)$ is one, which agrees with that induced by $\m^3\cup\m^2$. Hence we finish the proof in this subcase.
\pfd

\lm\label{ncell}
For any minimal subgraph $\m\subset\Z^n$ and any point $x\in \m$, $\m^n\neq\emptyset$ and there exists a positive constant $c_1=c_1(n),$ depending only on $n,$ such that
$$d(x,\m^n)\leq c_1.$$
\lmd

\pf
For any $x\in\m$, let $r$ be a positive integer with $\hat B_r(x)\cap\m^n=\emptyset.$ Then $\hat B_r(x)\cap\m\subset\de\m$. We can get a new graph $\m^\prime$ by removing $\hat B_{r-1}(x)\cap\m$ from $\m$.

Comparing $E_{\hat B_r(x)}\cap\partial\m$ with $E_{\hat B_r(x)}\cap\partial\m^\prime$, one can deduce that by the minimality of $\m$
\begin{align*}
2n|\de \hat B_r(x)\cap\m|&=2n(|\hat B_{r}(x)\cap\m|-|\hat B_{r-1}(x)\cap\m|)
\\&\geq|E_{\hat B_r(x)}\cap\partial\m^\prime|
\\&\geq|E_{\hat B_r(x)}\cap\partial\m^|
\\&\geq|\hat B_{r-1}(x)\cap\m|.
\end{align*}

This gives $$|\hat B_{r}(x)\cap\m|\geq (1+\frac{1}{2n})|\hat B_{r-1}(x)\cap\m|.$$

Therefore, we obtain that $$|\hat B_{r}(x)\cap\m|\geq(1+\frac{1}{2n})^r.$$

On the other hand, since $\m\subset\Z^n,$ $$|\hat B_{r}(x)\cap\m|\leq (1+2r)^n.$$ Hence we have $$r\leq c_1(n),$$ for some positive constant $c_1(n)$ depending only on $n.$

\pfd

\co\label{largencell}
Given any minimal subgraph $\m\subset\Z^n$ and any point $x\in \m$, let $\mathcal{C}(\m^n\cap \hat B_r(x))$ be the set of connected components of $\m^n\cap \hat B_r(x)$. Then $$\liminf\limits_{r\rightarrow \infty,A\in \mathcal{C}(\m^n\cap \hat B_r(x))}\dfrac{|\de A|}{|A|}=0.$$
In particular, $$\limsup\limits_{r\rightarrow \infty,A\in \mathcal{C}(\m^n\cap \hat B_r(x))}|A|=\infty.$$
\cod

\pf
The idea of the proof is similar to that of Lemma \ref{ncell}.

We argue by contradiction.
Then this yields $$\dfrac{|\de A|}{|A|}\geq c(n)>0,$$ for any $A\in \mathcal{C}(\m^n\cap \hat B_r(x))$, any $r\in \N$ and some positive constant $c(n)$ depending only on $n$. Hence we have
\begin{align}
c(n)|\m\cap\hat B_r(x)|\leq|\de\m\cap\hat B_r(x)|.\label{largencell1}
\end{align}

We can get a new graph $\m^\prime$ by removing $\hat B_{r-1}(x)\cap\m$ from $\m$. Comparing $\m$ with $\m^\prime$, using (\ref{largencell1}) and applying the minimality of $\m$, we have that
\begin{align*}
0&\geq|E_{\hat B_r(x)}\cap\partial\m|-|E_{\hat B_r(x)}\cap\partial\m^\prime|
\\&\geq |\de\m\cap\hat B_{r-1}(x)|-2n|\m\cap\de\hat B_{r}(x)|,
\end{align*}
and hence
\begin{align}
2n|\m\cap\de\hat B_r(x)|\geq|\de\m\cap\hat B_{r-1}(x)|\geq c(n)|\m\cap\hat B_{r-1}(x)|.\label{largencell2}
\end{align}
Note that $$|\m\cap\de\hat B_{r}(x)|=|\m\cap\hat B_{r}(x)|-|\m\cap\hat B_{r-1}(x)|,$$
which yields $$|\m\cap\hat B_r(x)|\geq (1+\frac{c(n)}{2n})^r,$$ for any $r\in\N$. But recall that $$|\m\cap\hat B_r(x)|\leq (2r+1)^n.$$ This is impossible for sufficiently large $r$.
\pfd

\lm\label{boundaryncell}
For any minimal subgraph $\m\subset\Z^n$ and any point $x\in \m$,
we have $$|(\m-\m^n)\cap \hat B_r(x)|\leq 2n|\m^n\cap N_1(\m-\m^{n})\cap\hat B_r(x)|,$$ where $N_1(\m-\m^{n}):=\{x\in\Z^n|d(x,\m-\m^{n})\leq 1\}.$
\lmd
\pf
Observe the new graph $\m^\prime$ obtained by removing $\m-\m^n$ from $\m\cap\hat B_r(x)$ and note that $\m-\m^n\subset\de\m$. Comparing $E_{\hat B_r(x)}\cap\partial\m^\prime$ with $E_{\hat B_r(x)}\cap\partial\m$ and recalling that $\m$ is minimal, we get that
\begin{align*}
|(\m-\m^n)\cap \hat B_r(x)|
&\leq |E_{\hat B_r(x)}\cap\partial\m|
\\&\leq|E_{\hat B_r(x)}\cap\partial\m^\prime|
\\&\leq 2n|\m^n\cap N_1(\m-\m^{n})\cap\hat B_r(x)|.
\end{align*}
\pfd

Next, we prove Theorem \ref{ncellroughlyisometry} and Theorem \ref{noboundedplane}.
\pf[Proof of Theorem \ref{ncellroughlyisometry}]
The natural embedding $\m^n\subset\m$ is surely a rough isometry by Lemma \ref{ncell}. Recall that $\deg(x)\leq 2n$, $|\hat B_{c_1}(x)|\leq (2c_1+1)^n$, and one has $|\hat B_{c_1}(x)\cap\m^n|\geq 1$ for any given $x\in\m$ by Lemma \ref{ncell}. One can show the inequality (\ref{ncellroughlyisometry1}) in Theorem \ref{ncellroughlyisometry} for sufficiently large $r$.

It is obvious that
$\m^n\cap N_1(\m-\m^{n})\subset\de\m^n$.

Hence we have
\begin{align*}
|(\m-\m^n)\cap \hat B_r(x)|&\leq 2n|\m^n\cap N_1(\m-\m^{n})\cap\hat B_r(x)|
\\&\leq 2n|\de \m^n\cap\hat B_r(x)|.
\end{align*}

By $\de\m=\de\m^n\cup(\m-\m^n)$, direct computation yields that $$\dfrac{|\de\m^{n}\cap\hat B_r(x)|}{|\de \m\cap\hat B_r(x)|}\geq \dfrac{1}{1+2n}.$$
\pfd

\pf[Proof of Theorem \ref{noboundedplane}]
Suppose it is not true, let $P$ be one of two parallel hyperplanes and $\bar\m$ be the usual distance projection image from $\m$ to $P$. Assume the distance of two parallel hyperplanes is $l$ and $\m^\prime$ is the new graph obtained by removing $\hat B_{r-1}(x)$ from $\m$ for any given $x\in \m$. Then one can show that
\begin{align}
&2|\bar\m\cap \hat B_r(x)|\leq |\de\m\cap \hat B_r(x)|\leq l|\bar\m\cap \hat B_r(x)|.\label{boundedbdy1}
\\&|\bar\m\cap \de\hat B_r(x)|\leq |\m\cap\de \hat B_r(x)|\leq l|\bar\m\cap\de \hat B_r(x)|.\label{boundedbdy2}
\end{align}

On the other hand, using (\ref{boundedbdy1}), (\ref{boundedbdy2}) and the minimality of $\m$, we have
\begin{align*}
0&\geq |E_{\hat B_r(x)}\cap\partial\m|-|E_{\hat B_r(x)}\cap\partial\m^\prime)|
\\&\geq 2|\bar\m\cap \hat B_{r-1}(x)|-2n|\m\cap\de \hat B_r(x)|
\\&\geq 2|\bar\m\cap \hat B_{r-1}(x)|-2nl|\bar\m\cap\de \hat B_r(x)|,\label{boundedbdy3}
\end{align*}
which yields that $$|\bar\m\cap \hat B_r(x)|\geq (1+\frac{1}{nl})^r.$$
But since $\bar\m\subset \Z^{n-1},$ $$|\bar\m\cap \hat B_r(x)|\leq (2r+1)^{n-1}.$$ This is a contradiction when $r$ is sufficiently large.
\pfd

{Before introducing a maximum principle property of minimal subgraphs, we adopt some notations.}

Let $p_i:\Z^n\longrightarrow\Z^{n-1}$ be the $i$-th projection for any integer $1\leq i\leq n$,
where
$$p_i(x_1,x_2,\cdots,x_n)=(x_1,x_2,\cdots,x_{i-1},x_{i+1},\cdots,x_n)$$ for any $x=(x_1,x_2,\cdots,x_n)\in\Z^n.$

For any subgraph $\m\su\Z^n$ and any $\hat x=(x_1,x_2,\cdots,x_{n-1})\in p_i(\m)$, we define $$H_i(\hat x,\m):=\sup\limits_{y=(y_1,\cdots,y_n)\in\m,\ p_i(y)=\hat x}y_i$$
and
$$h_i(x,\m):=\inf\limits_{y=(y_1,\cdots,y_n)\in\m,\ p_i(y)=\hat x}y_i.$$
Note that $H_i(\hat x,\m),h_i(\hat x,\m)$ may be infinite.

\prp\label{maximumprinciple}
Suppose that $\m\su\Z^n$ is minimal and given two finite subsets $\Omega_1,\Omega_2:=\{x\in\Z^{n-1}|d(x,\Omega_1)\leq 1\}\su p_i(\m)\su\Z^{n-1}$, such that the geometric realization of $\Omega_1$ in $\R^{n-1}$ is the closure of a bounded open domain. Then neither
\begin{equation}
\max\limits_{\hat x\in\Omega_2-\Omega_1} H_i(\hat x,\m)<H:=\max\limits_{\hat x\in\Omega_1} H_i(\hat x,\m)<+\infty,\label{max}
\end{equation}
nor
\begin{align}
\min\limits_{\hat x\in\Omega_2-\Omega_1} h_i(\hat x,\m)>h:=\min\limits_{\hat x\in\Omega_1} h_i(\hat x,\m)>-\infty.\label{min}
\end{align}
holds for any integer $1\leq i\leq n$.
\prpd

\pf
It suffices to prove (\ref{max}) fails for $i=1$. We argue by contradiction.

Consider the nonempty set $\Omega:=\{\hat x\in\Omega_2:H(\hat x,\m)=H\}$ and the set $\tilde\Omega:=\{x\in\Z^n:p_1(x)\in\Omega,x_1=H\}\su\{H\}\times\Z^{n-1}$. Note that $\Omega\su\Omega_1$ and $\tilde\Omega\su\m$, by the assumption (\ref{max}). Now one can obtain a new graph $\m^\prime$ from $\m$ by deleting the finite subset $\tilde\Omega$.

Using (\ref{max}) again, one easily deduces that $$|E_B\cap\partial\m)|-|E_B\cap\partial\m^\prime|\geq |\hat\partial\tilde\Omega|>0,$$ where $B$ is any finite ball containing $\tilde\Omega$ and $\hat\partial\tilde\Omega$ is the boundary edge of $\tilde\Omega$ in $\{H\}\times\Z^{n-1}$. This is impossible by the minimality of $\m$.


\pfd

\section{Open problems}
In this section, let $\m$ always be a minimal subgraph in $\Z^n$. We propose some open problems on geometry and topology of high-dimensional minimal subgraphs.

The first problem is the stronger version of Theorem \ref{ncellroughlyisometry}.
\pr\label{probelm1}
Are the following true:
\begin{align*}
\lim\limits_{r\rightarrow\infty}\dfrac{|\m^n\cap\hat B_r(x)|}{|\m\cap\hat B_r(x)|}=1 \quad and\quad
\lim\limits_{r\rightarrow\infty}\dfrac{|\de\m^n\cap\hat B_r(x)|}{|\de\m\cap\hat B_r(x)|}=1?
\end{align*}
\prd

The second problem is about the number of connected components and some ``big" connected component of any minimal subgraph.
\pr\label{problem2}
Do $\m^n$ and $\m$ have only finitely many connected components? Does at least one of these connected components contains a subgraph isometric to $\underbrace{\Z_{\geq0}\times\cdots\times\Z_{\geq0}}_{n\ \emph {times}}?$
\prd

Motivated by positive density at infinity of any minimal hypersurface in $\R^n$, it is natural to ask the growth rates of any minimal subgraph and its boundary.
\pr\label{problem3}
Whether do there exist positive constants $c_2=c_2(n),c_3=c_3(n),c_4=c_4(n)\ and\ c_5=c_5(n)$ depending only on $n$ such that
\begin{align*}
c_2r^n\leq|\m\cap\hat B_r(x)|\leq c_3r^n\ \ and\ \
c_4r^{n-1}\leq|\de\m\cap\hat B_r(x)|\leq c_5r^{n-1}?
\end{align*}
\prd
\rmk
$\m$ can be replaced by $\m^n$ by Lemma \ref{ncell} and Lemma \ref{boundaryncell}.
\rmkd
Note that the monotonicity formula of any minimal subgraph or its boundary does not hold in general; see Figure \ref{fig2-1}.
\ Proposition \ref{uppergrowth} gives the upper bounds in Problem \ref{problem3}. So that it suffices to consider the lower bounds in Problem \ref{problem3}.

The last question is as follows.
\pr\label{problem4}
For a minimal subgraph $\m\subset\Z^n$, do bounded connected components of $\m^n$ exist\emph{?} How to characterize the bounded connected components of $\m^n$ if they exist\emph{?} 
\prd

\textbf{Acknowledgements.} We thank Florentin M\"unch for helpful discussions and suggestions on the 1-Laplaican and the definition of area-minimizing subgraphs in $\Z^n.$ B.H. is supported by NSFC, no. 12371056 and Shanghai Science and Technology Program [Project No. 22JC1400100]. Z.H. is supported by NSFC, no.12301094 and Guangzhou Basic and Applied Basic Research Foundation No. 2024A04J3483.

\textbf{Data Available Statement.} The data used to support the findings of this study are available from the corresponding author upon request.





\bibliographystyle{alpha}
\bibliography{1-Laplace}

\end{document}